\documentclass[preprint,11pt]{article}

\usepackage[english]{babel}	
\usepackage[utf8x]{inputenc}		
\usepackage{enumerate}
\usepackage{textcomp}                
\usepackage{typearea}
\usepackage{pdfpages}
\usepackage[toc]{appendix}
\usepackage[]{todonotes}
\usepackage{subcaption}

\usepackage{amsfonts, amsmath, amssymb, amsthm, dsfont}
\usepackage{amsthm}			
\usepackage{thumbpdf}	
\usepackage{natbib}
\usepackage{booktabs}

\usepackage{pgf, tikz}
\usepackage{tikz-cd}		
\usepackage{bbm}

\theoremstyle{definition}
\newtheorem{definition}{Definition}[section]

\newtheorem*{assumption*}{Assumption}

\newtheorem*{condition*}{Condition}

\theoremstyle{plain}
\newtheorem{theorem}[definition]{Theorem}
\newtheorem{proposition}[definition]{Proposition}
\newtheorem{lemma}[definition]{Lemma}
\newtheorem{cor}[definition]{Corollary}
\theoremstyle{remark}
\newtheorem{remark}{Remark}

\usepackage{graphicx}
\usepackage{hyperref}

\newcommand{\N}{\mathbb{N}}

\newcommand{\R}{\mathbb{R}}

\newcommand{\E}{\mathbb{E}}

\newcommand{\pconv}{\xrightarrow{P}}

\newcommand{\Cov}{\operatorname{Cov}}

\newcommand{\wconv}{\Rightarrow}

\newcommand{\trunc}{\xi}
\newcommand{\D}{\mathrm{D}}
\newcommand{\diag}{\mathrm{diag}}
\newcommand{\f}{\boldsymbol{f}}
\newcommand{\I}{\boldsymbol{I}}
\newcommand{\sgn}{\mathrm{sign}}
\newcommand{\im}{\boldsymbol{i}}

\title{Rate-optimal estimation of the Blumenthal--Getoor index of a Lévy process}
\author{Fabian Mies\footnote{RWTH Aachen University, Institute of Statistics,
		mies@stochastik.rwth-aachen.de}}

\begin{document}

\maketitle

\begin{abstract}
	The Blumenthal--Getoor (BG) index characterizes the jump measure of an infinitely active Lévy process. 
	It determines sample path properties and affects the behavior of various econometric procedures. 
	If the process contains a diffusion term, existing estimators of the BG index based on high-frequency observations only achieve rates of convergence which are suboptimal by a polynomial factor.
	In this paper, a novel estimator for the BG index and the successive BG indices is presented, attaining the optimal rate of convergence.
	If an additional proportionality factor needs to be inferred, the proposed estimator is rate-optimal up to logarithmic factors.
	Furthermore, our method yields a new efficient volatility estimator which accounts for jumps of infinite variation. 
	All parameters are estimated jointly by the generalized method of moments.
	A simulation study compares the finite sample behavior of the proposed estimators with competing methods from the financial econometrics literature.
	
	\vspace{1ex}\noindent\textbf{Keywords:} high-frequency; method of moments; jump activity; Fisher information; non-diagonal rate matrix; asymptotic distribution;\\[1ex]
\end{abstract}

\section{Introduction}\label{sec:intro}

Models for continuous time stochastic processes with jumps have gained increased interest in the statistical literature,
most prominently in financial econometrics where they are used as a model for asset prices \citep{andersen2002empirical,christensen2014fact}.
The jump behavior of these processes $X_t$ can be broadly characterized in terms of the jump activity index, given by 
\begin{align}
	\alpha = \inf \left\{p: \sum_{s\leq T} |\Delta X_s|^p<\infty\right\}. \label{eqn:JA-def-1}
\end{align}
Here, $\Delta X_s = X_s - X_{s-}$ denotes the size of a jump at time $s$.
If $X_t$ is a Lévy process, $\alpha$ is also known as the Blumenthal-Getoor index \citep{blumenthal1961sample}.
The index $\alpha$ depends on the small jumps only, and for semimartingales, its range is $\alpha\in[0,2]$.
Various qualitative properties of the process $X_t$ can be expressed in terms of the jump activity index.
If the process has only finitely many jumps in total, then $\alpha=0$, and if the jumps are of finite variation, we have $\alpha\leq 1$. 
Conversely, $\alpha<1$ implies jumps of finite variation. 
Furthermore, the value of $\alpha$ has implications for various econometric procedures.
For example, if the jumps are treated as a nuisance, jump-robust estimation of integrated volatility requires $\alpha<1$ \citep{jacod2014remark}, as well as an efficient drift estimator due to \cite{gloter2018jump}.
In these applications, a higher jump activity typically induces a non-negligible bias which can not be easily corrected if the jumps are considered as a nuisance.
Hence, highly active jumps need to be modeled more explicitly, as done by \cite{Amorino2018} for drift estimation, and by \cite{jacod2014efficient, jacod2016efficient} for volatility estimation.

As the jump activity index is a central property of infinite activity jump models, it is natural to consider statistical estimation of its precise value.
Recent interest in this topic has been initiated by \cite{ait2009estimating}, who study the estimation of $\alpha$ based on discrete high-frequency observations $X_{i/n}, i=1,\ldots, n$, where $X$ is an Itô semimartingale with a non-vanishing diffusion component. 
They specify \eqref{eqn:JA-def-1} more precisely by defining $\alpha$ in terms of the spot jump compensator $\nu_t$, assuming that $\nu_t\left((-x,x)^c\right) = r_t |x|^{-\alpha} + \mathcal{O}(|x|^{\delta-\alpha})$ as $|x|\to 0$ for a predictable process $r_t$, and some $\delta>0$.  
The statistical challenge is that, based on discrete observations at a given frequency, the small jumps can hardly be distinguished from the continuous diffusion movement. 
The solution of \cite{ait2009estimating} is to introduce a threshold sequence $\tau_n \propto h_n^\omega\to 0$ and consider 
\begin{align}
	U( \tau_n) = \sum_{i=1}^n \mathds{1}\left(\left|X_{\frac{i}{n}} - X_{t_\frac{i-1}{n}}\right| > \tau_n\right). \label{eqn:AJ-U}
\end{align}
If $\omega < 1/2$, the contribution of the diffusion towards the statistic $U(\tau_n)$ will be negligible. 
The jump activity can be identified via the approximate scaling relation $U(\tau_n) \propto \tau_n^{-\alpha}$, and \cite{ait2009estimating} show that this approach lends itself to derive an estimator of $\alpha$ with rate of convergence $n^{\alpha/10}$. 
Replacing the indicator in \eqref{eqn:AJ-U} by a suitable smooth function, \cite{jing2012jump} improve this rate to $n^{\alpha/8}$.
So far, the best rates have been achieved by \cite{reiss2013testing} for the case that $X_t$ is a Lévy process, and by \cite{bull2016near} for Itô semimartingales. 
Both authors construct estimators which converge at rate $n^{\alpha/4 - \epsilon}$ for arbitrary $\epsilon>0$. 
In both cases, the precise form of the estimator depends on the desired rate defect $\epsilon>0$. 

In the considered high-frequency setting, the optimal rate of convergence for estimating $\alpha$ is conjectured to be $n^{\alpha/4}$, up to logarithmic factors. 
This lower bound is justified by the results of \cite{ait2012identifying}, who study the diagonal entries of the Fisher matrix of a fully parametric submodel consisting of the sum of a Brownian motion and a symmetric $\alpha$-stable Lévy motion. 
A matching LAN result is not available since the off-diagonal entries have not been studied. 
This lower bound is discussed in Section \ref{sec:estimator-lb}.
It should be highlighted that the achievable rate of convergence for estimating $\alpha$ depends on whether the process contains a non-vanishing diffusion component. 
If we consider a pure-jump Itô semimartingale, the jump activity index can be estimated at rate $\sqrt{n}$ based on high-frequency observations \citep{todorov2015jump}.

Although the estimators of \cite{reiss2013testing} and \cite{bull2016near} almost achieve the optimal rate of convergence, there is so far no procedure which attains the $n^{\alpha/4}$ lower bound, even in the case where $X_t$ is a Lévy process.
This issue has also been formulated as an open problem by \cite{reiss2013testing}.
In this paper, we propose a new estimator of $\alpha$ for the Lévy case.
If only $\alpha$ is unknown, the estimator achieves the optimal rate of convergence, matching the lower bound of \cite{ait2012identifying}.
If an additional proportionality factor $r$ needs to be estimated, our estimator is rate-optimal up to a factor of $\log n$ for both $r$ and $\alpha$. 
Furthermore, we show that the diagonally rescaled Fisher matrix in the submodel considered by \cite{ait2012identifying} is asymptotically singular for the combined parameter $(\alpha,r)$, and hence we conjecture that our rate of convergence is in fact optimal.
Our procedure also yields an efficient estimator of the volatility $\sigma^2$ of the diffusion component of $X_t$ in the presence of jumps of infinite variation.
Under analogous conditions on the jump behavior, \cite{jacod2014efficient, jacod2016efficient} have derived a different efficient estimator of volatility which is robust to highly active jumps.
Hence, our estimator is an alternative to the method of \cite{jacod2014efficient}, although the latter is valid for Itô semimartingales and we restrict our attention to Lévy processes.
The proposed estimator is based on the generalized method of moments, and we estimate the jump and the diffusion parameters jointly in a single step as the solution of a system of estimating equations.

Our model allows for an asymmetric behavior of the small jumps. In particular, for a Lévy process $X_t$ with characteristic triplet $(\mu,\sigma^2,\nu)$, we suppose that the Lévy measure $\nu$ is locally stable in the sense that, for $z$ close to $0$,
\begin{align}
	\nu(dz) \approx \tilde{\nu}(dz)  = \sum_{m=1}^M \frac{\alpha_m}{|z|^{1+\alpha_m}} \left( r_m^+\mathds{1}_{z>0} + r_m^- \mathds{1}_{z<0} \right) \, dz. \label{eqn:def-stablesum-1}
\end{align}
Here, $M$ is a natural number, $r_m^\pm\geq 0$, $m=1,\ldots, M$, and the $0<\alpha_M < \ldots < \alpha_1 <2$ are the successive Blumenthal-Getoor indices, as introduced by \cite{ait2012identifying}.
The approximation in \eqref{eqn:def-stablesum-1} will be made precise in the sequel.
In particular, the BG index of $X_t$ will be $\alpha=\alpha_1$. 
We construct an estimator for the parameter vector $\theta\in\R^{3M+1}$ consisting of the volatility $\sigma^2$, the indices $\alpha_m$, and the proportionality factors $r_m^\pm$.

The remainder of this paper is structured as follows. 
In Section \ref{sec:estimator}, we present our model and the proposed estimator. 
A central limit theorem is given, establishing the rate $n^{\alpha/4}$. 
The rate of convergence and related lower bounds are discussed in Section \ref{sec:estimator-lb}. 
By means of a simulation study (Section \ref{sec:MC}), we compare the finite sample properties of our method with the jump activity estimators of \cite{bull2016near, reiss2013testing} and the volatility estimator of \cite{jacod2014efficient}.
All technical results, which might be of independent interest, are outlined in Section \ref{sec:technical-levy}, and the detailed proofs are gathered in Section \ref{sec:proofs-levy}.

\subsection{Notation}
For two real numbers $a,b$, we denote $a\wedge b = \min(a,b)$, $a\vee b = \max(a,b)$. The indicator function of a set $A$ is denoted as $\mathds{1}_A$. For a function $f=f(a,b,\ldots)$, $\partial_a f$ denotes the partial derivative w.r.t.\ $a$, and for a function $f(\theta)\in\R^m$ with $\theta\in\R^k$, the gradient matrix is denoted by $(\D_\theta f)_{j,l} = \partial_{\theta_l}f_j$. For $\delta>0$, $B_\delta(0)$ is the ball around $0$ with radius $\delta$ in $\R^k$, where $k$ is evident from the context. $\I_d\in\R^{d\times d}$ denotes the identity matrix. The multivariate normal distribution with covariance matrix $\Sigma$ and mean $0$ is denoted as $\mathcal{N}(0,\Sigma)$, and $\wconv$ denotes weak convergence of probability measures resp.\ random elements. The expectation operator is $\E$, and dependence upon a parameter $\theta$ is denoted as $\E_\theta$.

\section{Model and estimator}\label{sec:estimator}

Consider a univariate Lévy process $X_t$, $X_0=0$, with characteristic triplet $(\mu,\sigma^2, \nu)$ for a drift parameter $\mu\in\R$, volatility parameter $\sigma^2>0$, and a Lévy measure $\nu$, i.e.\ $\int (1\wedge |z|^2)\,\nu(dz)<\infty$. 
We choose an odd truncation function $\trunc$ such that $|\trunc|\leq 2$ and $\trunc(z)=z$ for $z\in(-1,1)$. 
Then $X_t$ admits the Lévy-Itô decomposition 
\begin{align}
	X_t = \mu t + \sigma B_t + \int_0^t \int (z-\trunc(z)) \, N(dz, ds) + \int_0^t \int \trunc(z)\,(N(dz, ds) - \nu(dz)\otimes ds), \label{eqn:levy-ito}
\end{align}
where $N(dz, ds)$ is a Poisson point process with intensity measure $\nu(dz)\otimes ds$, and $B_t$ is a standard Brownian motion, independent of $N$. 
The value of $\mu$ depends on the choice of the truncation function $\trunc$, but for our purposes, it will turn out that $\mu$ is negligible anyways. 
To make the approximation \eqref{eqn:def-stablesum-1} precise, we suppose that 
\begin{align}
	\begin{split}
		\left|\nu([x,\infty)) - \tilde{\nu}([x,\infty))\right| &\leq L |x|^{-\rho}, \quad x\in(0,1], \\
		\left|\nu(-\infty,x]) - \tilde{\nu}(-\infty,x])\right| &\leq L |x|^{-\rho}, \quad x\in[-1, 0), 
	\end{split} \label{eqn:def-locstable}
\end{align}
for some $L>0$ and $\rho>0$. 
The approximating measure $\tilde{\nu}$ is given by the Lebesgue density 
\begin{align}
	\tilde{\nu}(dz)  = \sum_{m=1}^M \frac{\alpha_m}{|z|^{1+\alpha_m}} \left( r_m^+\mathds{1}_{z>0} + r_m^- \mathds{1}_{z<0} \right) \, dz, \label{eqn:def-stablesum}
\end{align}
for some natural number $M$ and parameters $\boldsymbol{\alpha} = (\alpha_1,\ldots, \alpha_M)\in (0,2)^M$, and $\boldsymbol{r} = (r_1^+, r_1^-,\ldots, r_M^+, r_M^-)\in \R_{\geq 0}^{2M}$. 
The remainder term in \eqref{eqn:def-locstable} is treated as a nuisance. 
In particular, this remainder may still consist of infinite activity jumps. 
Our main result will require $\rho<\alpha_M$, such that the nuisance jumps are in a sense less active than the Lévy measure $\tilde{\nu}$ and asymptotically negligible.
The parameters of the modeled part are summarized as
\begin{align}
	\theta = (\sigma^2, \alpha_1, r_1^+, r_1^-,\ldots, \alpha_M, r_M^+, r_M^-)\in\Theta\subset\R^{3M+1}. \label{eqn:parameter-vector}
\end{align}
where $\Theta$ contains all parameter vectors $\theta$ as specified, such that additionally 
\begin{align*}
	\alpha=\alpha_1> \alpha_2 > \ldots > \alpha_M > \frac{\alpha}{2},\qquad r_m^++r_m^->0,\, i=1,\ldots, M,\quad \sigma^2\geq 0.
\end{align*}
The value $\alpha=\alpha_1$ is of central importance.
In particular, we need to impose the lower bound $\alpha_M>\alpha/2$ to ensure identifiability of the full parameter vector $\theta$, see \cite{ait2012identifying}. 
Note that the definition \eqref{eqn:def-stablesum} is the same as given by \cite{jacod2016efficient} for the symmetric case.

In the high-frequency sampling setting considered here, we are given $n$ observations $X_{ih_n}$, $i=1,\ldots, n$ with observation frequency $h_n\to 0$ such that $nh_n = T$ is constant. 
Without loss of generality, let $T=1$ and $h = h_n = 1/n$. 
Equivalently, we observe the $n$ increments $\Delta_{n,i} X = X_{i h_n} - X_{(i-1)h_n} \sim X_{h_n}$, which constitute a triangular array of random variables with iid rows. 
The law of $X_{h_n}$ is not fully described by the parameters $(\sigma^2, \boldsymbol{r},\boldsymbol{\alpha})$ due to the remainder in \eqref{eqn:def-locstable}.
Hence, we approximate it by a fully specified Lévy process $\tilde{Z}_t$ with characteristic triplet $(0,\sigma, \tilde{\nu})$. 
The process $\tilde{Z}_t$ may be represented as 
\begin{align*}
	\tilde{Z}_t = \sigma B_t + \sum_{m=1}^M S_t^m,
\end{align*}
where $B_t, S_t^m$, $m=1,\ldots, M$, are independent Lévy processes, $B_t$ is a standard Brownian motion, and the $S_t^m$ are skewed $\alpha_m$-stable process with Lévy measure $|z|^{-1-\alpha_m} (r_m^+ \mathds{1}_{z>0} + r_m^- \mathds{1}_{z<0})$.

We suggest to estimate the parameter $\theta$ via the method of moments. 
In particular, we choose $3M+1$ functions $f_j:\R\to\R$, $\f = (f_1, \ldots, f_{3M+1})$, and a suitable scaling factor $u=u_n$, and define $\hat{\theta}=\hat{\theta}_n$ to be a solution of the equation 
\begin{align}
	F_n(\hat{\theta}_n) = \left[\frac{1}{n} \sum_{i=1}^n \f(u_n \Delta_{n,i}X) \right]- \E_{\hat{\theta}_n} \f(u_n \tilde{Z}_{h_n})\overset{!}{=}0. \label{eqn:def-GMM}
\end{align}
Here and in the following, $\E_\theta f(\tilde{Z}_h)$ denotes the expectation such that $\tilde{Z}_h$ is determined by the parameter vector $\theta$.
Since $\tilde{Z}_h$ is a fully parametric approximation of $X_h$, the function $F_n(\theta)$ can be be computed numerically, such that $\hat{\theta}_n$ is a feasible estimator.
To distinguish a generic parameter value from the parameters governing $X_t$, we denote by $\theta_0$ the true parameter such that \eqref{eqn:def-locstable} holds.

To study the limit of $\hat{\theta}_n$, we employ the standard framework for estimating equations as reviewed by \cite{jacod2017review}. 
Under the assumptions imposed below, we show that $\hat{\theta}_n-\theta_0 \approx - (\D_\theta F_n(\theta_0))^{-1} F_n(\theta_0)$, up to negligible terms.
In order for $\hat{\theta}_n$ to have good asymptotic properties, the choices of the moment functions $\f$ and the scaling factor $u_n$ are crucial. 
In particular, to derive a central limit theorem for $F_n(\theta_0)$ (see Lemma \ref{lem:moment-clt}), we need to control the sampling variance in \eqref{eqn:def-GMM} as well as the bias incurred by approximating $X_t$ by $\tilde{Z}_t$.
Furthermore, the asymptotic behavior of $\D_\theta F_n(\theta)$ as $n\to\infty$ needs to be treated (see Lemma \ref{lem:expect-derivative}).
To this end, the following properties turn out to be sufficient. 

\begin{condition*}[F1]\label{ass:F1}
	For $j=1,\ldots, 3M+1$, the functions $f_j\in\mathcal{C}^3(\R)$ satisfy $\|f_j^{(k)}\|_\infty<\infty$ for $k=0,1,2,3$, and $f_j'\in L_1(\R)$. 
\end{condition*}

The smoothness imposed by Condition \nameref{ass:F1} is used to bound the bias incurred by approximating $\E \f(uX_{h_n})$ by $\E_\theta \f(u\tilde{Z}_{h_n})$, see Corollary \ref{cor:bias-complete} below.
To control the sampling variance, we do not only require smoothness of the employed moment functions, but they further need to be of a specific shape.

\begin{condition*}[F2]\label{ass:F2}
	The function $f_1$ is symmetric and satisfies $f_1(0)=f_1'(0)=0 \neq f_1''(0)$. The functions $f_j$, $j=2,\ldots, 3M+1$, are identically zero on the interval $[-\eta,\eta]$ for some $\eta>0$.
\end{condition*}

Additional identifiability conditions are specified in assumption \nameref{ass:I} below.
The first moment function $f_1$ is approximately quadratic near zero, and will serve to identify the volatility $\sigma^2$. 
The functions $f_j(x)$ are smooth thresholds, which distinguish the diffusion from the jump component.
An example of suitable moment functions is given in section \ref{sec:MC}. 
To ensure that the threshold is effective, we require that $u_nX_{h_n}\to 0$ in probability, i.e.\ $u_n = o(\sqrt{n})$.
By choosing an appropriate scaling sequence as follows, the moments $\E f_j(u_n\tilde{Z}_{h_n})$, $j\geq 2$, will be dominated by the jump component.  

\begin{condition*}[U]\label{ass:U}
	$u_n\to\infty$ such that $u_n = \frac{\tau\sqrt{n}}{\sqrt{\log n}}$ for some $\tau < \frac{\eta}{\sigma \sqrt{8}}$.
\end{condition*}

Although potentially not sharp, the upper bound on the factor $\tau$ is required to derive our asymptotic result. 
For details, see the technical Lemma \ref{lem:moments} below and the subsequent discussion.  
When choosing $u_n$ in accordance with condition \nameref{ass:U}, it suffices to use a reasonable upper bound on $\sigma$.
Furthermore, the simulation results presented in section \ref{sec:MC} show that larger values of $u_n$ also perform well in finite samples.

To formulate our main result on the asymptotic behavior of $\hat{\theta}$, we introduce the quantities
\begin{align*}
	\mathcal{J}_\alpha^\pm g (x) &= \alpha \int  \frac{g(x+z)-g(x) - g'(x) \trunc(z)}{|z|^{1+\alpha}} \mathds{1}_{\{\pm z >0\}}\, dz, \qquad \alpha\in(0,2),
\end{align*}
which exist if $\|g\|_\infty, \|g''\|_\infty <\infty$. 
Furthermore, we introduce the matrices
\begin{align*}
	\gamma_{n,m}(\theta) &= 
		\begin{pmatrix}
		  1 			 & 0  & 0 \\
		  -r_m^+ \log u_n  & 1  & 0 \\
		  -r_m^- \log u_n  & 0  & 1
		\end{pmatrix},\quad m=1,\ldots, M,\\
	\Gamma_n(\theta)  &= \diag(\I_1, \gamma_{n,1},\ldots, \gamma_{n, M}) \in \R^{(3M+1) \times (3M+1)}, \\
	\bar{\Lambda}_n(\theta) & = \sqrt{h_n}\,\diag(\sqrt{h_n}^{-1},  u_n^{\alpha_1 - \frac{\alpha_1}{2}}, u_n^{\alpha_1 - \frac{\alpha_1}{2}},  u_n^{\alpha_1 - \frac{\alpha_1}{2}}, u_n^{\alpha_2 - \frac{\alpha_1}{2}},\ldots \\
	&\qquad\qquad \ldots,  u_n^{\alpha_{M-1}-\frac{\alpha_1}{2}}, u_n^{\alpha_M - \frac{\alpha_1}{2}}, u_n^{\alpha_M - \frac{\alpha_1}{2}}, u_n^{\alpha_M - \frac{\alpha_1}{2}}),
\end{align*}
and the matrix $A(\theta)\in\R^{(3M+1)\times(3M+1)}$, given by
\begin{align*}
		A(\theta)_{1,1} = f_1''(0)/2, \quad A(\theta)_{1,j}=A(\theta)_{j,1}=0, j\neq 1,
\end{align*}
and for $m=1,\ldots, M$, $j=2,\ldots, 3M+1$,
\begin{align*}
	A(\theta)_{j, 3m-1} &= \partial_{\alpha_m} (r_m^+\mathcal{J}_{\alpha_m}^+f_j(0) + r_m^-\mathcal{J}_{\alpha_m}^-f_j(0)), \\
	A(\theta)_{j, 3m}   &= \mathcal{J}_{\alpha_m}^+f_j(0), \qquad 
	A(\theta)_{j, 3m+1} = \mathcal{J}_{\alpha_m}^-f_j(0).
\end{align*}
These derivatives exist because $\|f\|_\infty, \|f''\|_\infty$ are finite.
Finally, we introduce the symmetric positive semidefinite matrix $\Sigma(\theta)$ given by 
\begin{align*}
	\Sigma(\theta)_{1,1} &= \frac{\sigma^4 f_1''(0)^2}{2}, \\
	\Sigma(\theta)_{1,j} & =\Sigma(\theta)_{j,1}=0, \qquad j\geq 2,\\
	\Sigma(\theta)_{j,k} &= \left(r_1^+\mathcal{J}_{\alpha_1}^++ r_1^-\mathcal{J}_{\alpha_1}^- \right)(f_j\cdot f_k)(0),\qquad j,k\geq 2.
\end{align*}
If clear from the context, we will omit the dependence on $\theta$.
Using this notation, we can formulate the remaining identifiability condition.

\begin{condition*}[I]\label{ass:I}
	For the true parameter $\theta_0$, $A(\theta_0)$ is regular.
\end{condition*}

\begin{remark}
Analyzing the degrees of freedom of the equation $|A(\theta)|=0$ suggests that condition \nameref{ass:I} is, in fact, the generic case.
To demonstrate this point, we construct a set of moment functions satisfying the identifiability condition. 
Consider the case $M=1$ with $\alpha_m=\alpha$ and $r_m^\pm=r^\pm$, $m=1$. 
We can construct a set of moment functions satisfying condition \nameref{ass:I} as follows.
Let $f_1=f$ and $g$ be symmetric functions satisfying conditions \nameref{ass:F1} such that $f_1''(0)\neq 0$, and $g$ vanishes on $[-1,1]$.
Furthermore, denote $a=\mathcal{J}_\alpha^+ g(0)=\mathcal{J}_\alpha^- g(0)$, and $b=\partial_{\alpha} \mathcal{J}_\alpha^\pm g(0)$. 
We set $f_2(x)=g(x), f_3(x) = g(2x)$, and $f_4(x) = g(x)\mathds{1}_{x>0}+g(2x)\mathds{1}_{x<0}$.
Note that $\mathcal{J}^\pm f_3 (0)= 2^\alpha \mathcal{J}^\pm g(0)=2^\alpha a$, as well as $\mathcal{J}^+f_4(0) = a$, and $\mathcal{J}^- f_4(0) = 2^\alpha a$.
Then one can check that
\begin{align*}
	A(\theta_0)&=A(\sigma^2, r^+, r^-,\alpha) \\
	&=\begin{pmatrix}
	f''(0)/2 & 0&0 & 0\\
	0& (r^++r^-)b		& a& a &\\
	0& 2^\alpha(r^++r^-)(b+a\log 2)	& 2^\alpha a& 2^\alpha a &\\
	0&r^+b+r^-a 2^\alpha \log 2 &  a & a(1+2^\alpha)&
	\end{pmatrix},
\end{align*} 
with determinant $\det(A)=-\frac{f''(0)}{2}(r^++r^-)\, a^3 \,2^\alpha \log 2$.
Hence, $A(\theta_0)$ is regular for $(r^++r^-)>0$ and all $\alpha\in(0,2)$ if $g$ is chosen such that $a\neq 0$.
This is in particular the case for the choice of the moment functions for the simulation study in section \ref{sec:MC}.
\end{remark}

The main result of this paper is the consistency and asymptotic normality of $\hat{\theta}_n$, as summarized by the following theorem.

\begin{theorem}\label{thm:estimating-clt}
	Let $X_t$ be a Lévy process satisfying \eqref{eqn:def-locstable} with some $\rho<\alpha/2$, and parameter vector $\theta_0\in\Theta$. 
	Let $\f$ satisfy assumptions \nameref{ass:F1} and \nameref{ass:F2}, 
	and be such that $A(\theta_0)$ is regular, and let $u_n\to\infty$ be chosen according to \nameref{ass:U}.
	Then there exists a sequence of random vectors $\hat{\theta}_n$ solving \eqref{eqn:def-GMM}, such that $\hat{\theta}_n\to\theta$ in probability as $n\to\infty$. 
	This sequence is eventually unique, and, as $n\to\infty$,
	\begin{align*}
		\sqrt{n} A(\theta_0) \bar{\Lambda}_n \Gamma_n^{-1}(\theta_0) (\hat{\theta}_n-\theta_0) \wconv \mathcal{N}(0, \Sigma(\theta_0) ).
	\end{align*}
\end{theorem}

The resulting rate of convergence for the BG index $\alpha=\alpha_1$ is thus found to be $(n\log n)^{\frac{\alpha}{4}}$, which improves upon existing estimators and matches the lower bound of \cite{ait2012identifying} up to logarithmic factors.
However, the rate matrix of Theorem \ref{thm:estimating-clt} is non-diagonal. 
The phenomenon of a non-diagonal rate matrix has also been observed in the pure jump case, i.e.\ $\sigma^2=0$, see \cite{brouste2018efficient}.
We further discuss this aspect and the resulting marginal rates of convergence for $\hat{\alpha}_m$ and $\hat{r}_m^\pm$ in the next section.
Nevertheless, the matrices $\Gamma_n^{-1}$, $A(\theta_0)$, and $\Sigma(\theta_0)$ are block-diagonal, such that the volatility estimator $\hat{\sigma}^2$ is asymptotically independent of the estimator of the jump part.

The presented central limit theorem also holds for the fully specified case without nuisance, i.e.\ $L=0$ in \eqref{eqn:def-locstable}.
Even in this parametric case, we find that a simple GMM estimator based on $3M+1$ fixed moment functions, corresponding to $u_n=1$, will not achieve the best rate of convergence.
A careful construction of the estimating equation \eqref{eqn:def-GMM} is thus not only required to handle the nuisance term, but also for the underlying parametric problem itself. 

The proposed estimator for $\alpha$ can be contrasted with existing methods in the literature.
In an earlier study, \cite{reiss2013testing} suggests a test procedure for the value of $\alpha$ based on a statistic $T^m_n$ with tuning parameter $m\in\N$. 
Therein, it is established that $T_n^m\to Q(\alpha)$ as $n\to\infty$ at rate $n^{\frac{\alpha}{4}-\epsilon(m)}$, and $\epsilon(m)\to 0$ as $m\to\infty$.
By inverting the function $Q$, this approach yields a near-optimal estimator for $\alpha$.
The statistics $T_n^m$ are constructed based on nonlinear sample moments as in \eqref{eqn:def-GMM}, where the $f_j$ are linear combinations of trigonometric functions, i.e.\ $f_j(x) = \sum_{k} w_{k,j} \exp(i\lambda_kx)$. 
Choosing the weights $w_{k,j}$ carefully such that $\sum_k w_{k,j} \lambda_k^{2p} = 0$ for $p=1,\ldots, m-1$, \cite{reiss2013testing} is able to reduce the variance of the corresponding sample moments.
The arbitrarily small defect in the rate of convergence $n^{\alpha/4-\epsilon(m)}$ derived therein is thus due to the sampling variance.
In contrast, by choosing the moment functions to vanish near zero according to Condition \nameref{ass:F2}, we obtain a smaller variance of the sample moments. 

An alternative estimator achieving the rate $n^{\alpha/4-\epsilon}$ is presented by \cite{bull2016near}, which also uses functions which vanish near zero.
Therein, the value $\E \f(u_n X_{h_n})$ is approximated by a finite series expansion, and extending this expansion reduces the rate defect $\epsilon$.  
In contrast, we use the approximation $\E \f(u_n X_{h_n})\approx\E\f(u_n \tilde{Z}_{h_n})$.
Although the latter value is not available in explicit form and needs to be determined numerically, this approach allows us to decrease the bias of the estimating equation further than by any finite series expansion. 
In particular, we only incur a bias due to approximating the Lévy measure of $X_t$, but not due to a discretization of the time evolution of the process.
Thus, our method effectively circumvents the variance issue of \cite{reiss2013testing} and the bias issue of \cite{bull2016near}.
This allows us to eliminate the polynomial rate defect and achieve a faster rate of convergence.

\section{Asymptotic optimality}\label{sec:estimator-lb}

It is natural to ask whether our proposed estimator is asymptotically optimal.
From Theorem \ref{thm:estimating-clt}, we find that \begin{align}
	\sqrt{n}(\hat{\sigma}^2_n - \sigma^2) \wconv \mathcal{N}(0, 2\sigma^4), \label{eqn:sigma-efficient}
\end{align}
which matches the optimal estimator in the situation without jumps. 
That is, $\hat{\sigma}^2_n$ is efficient.
In general, jumps of infinite variation reduce the achievable rate of convergence for volatility estimators \citep{jacod2014remark}.
Here, we are able to recover efficiency by modeling the infinite variation part of the jump measure explicitly via \eqref{eqn:def-locstable}.
The same methodology has been applied by \cite{jacod2014efficient, jacod2016efficient} to construct an efficient estimator of $\sigma^2$.
Note that the latter studies treat more general types of semimartingales, while we only derived a result for Lévy processes.
In contrast to the existing estimators, which use a multi-step debiasing procedure, we determine $\hat{\sigma}^2$ by a single set of estimating equations. 
While our approach is conceptually simple, solving the estimating equations \eqref{eqn:def-GMM} is computationally expensive.
A comparison of the finite sample performance is presented in Section \ref{sec:MC}.

As the asymptotic variance of the estimators $\alpha_m$ and $r_m^\pm$ depends on the choice of $\f$, they can not be expected to be variance efficient.
Furthermore, they are coupled via $\Gamma_n$ and via the matrix $A(\theta_0)$, which is in general dense. 
Inspecting the limit in Theorem \ref{thm:estimating-clt}, we find that 
\begin{align}
	\begin{split}
	\hat{\alpha}_m - \alpha_m 
	&= \mathcal{O}_P\left( u^{ \frac{\alpha_1}{2} - \alpha_m} \right) 
	 = \mathcal{O}_P\left( (n \log n)^{ \frac{\alpha_1}{4} - \frac{\alpha_m}{2}}\right),\\
	\hat{r}^\pm_m - r^\pm_m 
	&= \mathcal{O}_P\left( u^{ \frac{\alpha_1}{2} - \alpha_m} \log u \right) 
	 = \mathcal{O}_P\left( (n\log n)^{ \frac{\alpha_1}{4} - \frac{\alpha_m}{2}} \log(n)\right).
	\end{split} \label{eqn:rates-GMM}
\end{align}
To assess these rates of convergence, we may compare with the lower bound of \cite{ait2012identifying}. 
Therein, the authors compute the diagonal terms of the Fisher information $\mathcal{I}^n_\theta$ based on $n$ observations of $\tilde{Z}_{1/n}$ for the symmetric case $r_m^+=r_m^-=r_m$ and $M=2$. 
Their analysis of the diagonal entries $\mathcal{I}^n_{\alpha_m,\alpha_m}$ and $\mathcal{I}^n_{r_m,r_m}$ suggests that an asymptotically optimal estimator $(\hat{\alpha}_m^*, \hat{r}_m^*)$ should satisfy 
\begin{align}
	\begin{split}
		\hat{\alpha}_m^* - \alpha_m &= \mathcal{O}_P\left( (n\log n)^{\frac{\alpha_1}{4} - \frac{\alpha_m}{2}} / \log n \right),\\
		\hat{r}_m^* - r_m &= \mathcal{O}_P\left( (n\log n)^{\frac{\alpha_1}{4} - \frac{\alpha_m}{2}} \right).
	\end{split} \label{eqn:rates-diagonal}
\end{align}
Notably, even for $M=1$, the rates \eqref{eqn:rates-diagonal} are faster than \eqref{eqn:rates-GMM} by a logarithmic factor. 

This difference could potentially be explained by the neglected off-diagonal terms of $\mathcal{I}_\theta$. 
A similar phenomenon occurs in the pure jump case $\sigma^2=0$, $M=1$, where for any sequence of diagonal matrices $D_n$, the limit of $D_n \mathcal{I}^n_{(\alpha, r)} D_n$ is singular, see \cite[Thm.\ 3.4]{masuda2015parametric} and \cite[Thm.\ 2]{ait2008fisher}. 
Recently, \cite{brouste2018efficient} studied this case, and established the LAN property with a non-diagonal rescaling matrix $D_n$. 
They find that the optimal rate of convergence is slower than suggested by the diagonal entries of the Fisher matrix, by a factor of $\log n$.
A similar phenomenon is observed when estimating the Hurst parameter of a fractional Brownian motion based on high-frequency observations \citep{brouste2018local}.
There is no LAN result available for estimation of the BG index in the case $\sigma^2>0$, and a full investigation of the LAN property in the present case is out of scope of this paper. 
Nevertheless, we can adapt the proof of \cite{ait2012identifying} to unveil the off-diagonal entries $\mathcal{I}^n_{\alpha_1, r_1}$. 
It turns out that the diagonally rescaled Fisher matrix is asymptotically singular, just as in the pure-jump case.

\begin{proposition}\label{prop:fisher}
Let $\mathcal{I}^h$ denote the Fisher information matrix of $\tilde{Z}_h$ with $M=1$ and $\alpha_1=\alpha$, $r_1^+=r_1^-=r$. 
Then, as $h\to 0$,
\begin{align*}
	\frac{(h\log(1/h))^\frac{\alpha}{2}}{h} \begin{pmatrix}
		1 & 0 \\
		0 & \frac{1}{\log(1/h)}
	\end{pmatrix}
	\begin{pmatrix}
		\mathcal{I}_h^{r,r} & \mathcal{I}_h^{r,\alpha} \\ \mathcal{I}_h^{r,\alpha} & \mathcal{I}_h^{\alpha,\alpha}
	\end{pmatrix} 
	\begin{pmatrix}
		1 & 0 \\
		0 & \frac{1}{\log(1/h)}
	\end{pmatrix}
	\\
	\longrightarrow
	\frac{2r}{\sigma^\alpha (2-\alpha)^\frac{\alpha}{2}}
	\begin{pmatrix}
		\frac{1}{r^2} & \frac{1}{2r} \\
		\frac{1}{2r} &  \frac{1}{4}
	\end{pmatrix}.
\end{align*}
In particular, the limiting matrix is singular.
\end{proposition}

The diagonal entries of the Fisher information matrix should match the optimal rates of convergence in the case where only a single parameter is unknown, e.g.\ if $(\sigma^2, r_1^+, r_1^-)$ are known and $\alpha_1$ should be estimated. 
In this situation, a natural version of our estimator is to consider only a single moment function $f$. 
Analogous to \eqref{eqn:def-GMM}, for any $m\in\{1,\ldots, M\}$, we may estimate $\alpha_m$ as the solution of \begin{align}
	\tilde{F}_n(\alpha_m) = \frac{1}{n} \sum_{i=1}^n f(u_n \Delta_{n,i} X) - \E_\theta f(u_n\tilde{Z}_h) \overset{!}{=} 0.
\end{align}
With a slight abuse of notation, we may also estimate $r_m^\pm$ by the equation $\tilde{F}_n(r_m^\pm)=0$. 
To distinguish jumps and diffusion, we suppose $f$ satisfies the same conditions as $f_2,\ldots, f_{3M+1}$, i.e.\ it should vanish around zero. 

\begin{proposition}\label{prop:clt-single}
	Let $X_t$ be a Lévy process satisfying \eqref{eqn:def-locstable} with some $\rho<\alpha_1/2$, and parameter vector $\theta_0\in\Theta$. 
	Let $f$ be a non-negative function satisfying \nameref{ass:F1}, and $f(x)=0$ for $x\in [-\eta,\eta]$, and choose $u_n\to\infty$ such that \nameref{ass:U} holds.
	Fix some $m\in\{1,\ldots, M\}$, and suppose that $\mathcal{J}_{\alpha_m}^\pm f(0) >0$.
	Then there exists a consistent sequence of estimators $\hat{\alpha}_m$ satisfying $\tilde{F}_n(\hat{\alpha}_m)=0$, such that $\hat{\alpha}_m\to\alpha_m$ in probability as $n\to\infty$, and
	\begin{align*}
		u_n^{\alpha_m-\frac{\alpha_1}{2}} \log(u_n)\left(\hat{\alpha}_m - \alpha_m\right) \wconv \mathcal{N}\left(0, \frac{(r_1^+ \mathcal{J}_{\alpha_1} + r_1^-\mathcal{J}_{\alpha_1}) f^2(0)}{(r_m^+\mathcal{J}_{\alpha_m}^+ + r_m^-\mathcal{J}_{\alpha_m}^-)f(0)} \right).
	\end{align*}
	Under the same conditions, and if all parameters except for $r_m^+$ resp.\ $r_m^-$ are known, there exists a consistent sequence of estimators $\hat{r}_m^\pm$ solving $\tilde{F}_n(\hat{r}_m^{\pm})=0$ such that, as $n\to\infty$,
	\begin{align*}
		u_n^{\alpha_m-\frac{\alpha_1}{2}}\left(\hat{r}_m^\pm - r_m^\pm\right) \wconv \mathcal{N}\left(0, \frac{(r_1^+ \mathcal{J}_{\alpha_1} + r_m^-\mathcal{J}_{\alpha_1}) f^2(0)}{\mathcal{J}_{\alpha_m}^\pm f(0)} \right).
	\end{align*}
\end{proposition}

Since $u_n$ is of order $\sqrt{n/\log n}$, Proposition \ref{prop:clt-single} establishes precisely the rates \eqref{eqn:rates-diagonal}.
In the setting of \cite{ait2012identifying}, in particular $M=2$, this shows that $\hat{\alpha}_m$ resp.\ $\hat{r}_m^\pm$ are rate efficient if the remaining parameters $\theta$ are known. 
In contrast, if all parameters $\theta$ are unknown, $\hat{\theta}$ achieves the optimal rate of convergence, up to a logarithmic factor. 
Due to the singularity of the Fisher matrix, we conjecture that the achieved rates \eqref{eqn:rates-GMM} are in fact optimal.

\section{Simulation study}\label{sec:MC}

By means of a Monte Carlo study, we compare the finite sample performance of our estimator with the estimators of \cite{reiss2013testing} and \cite{bull2016near} for the Blumenthal-Getoor index $\alpha$, and with the volatility estimator of \cite{jacod2014efficient}.
To this end, we sample paths of a Lévy process $X_t$ given by 
\begin{align}
	X_t = B_t + S^{\alpha, \beta}_t + 0.1 S^{0.5, 0}_t. \label{eqn:levy-sim}
\end{align}
We denote by $S^{\alpha,\beta}_t$ the $\alpha$-stable Lévy motion with skewness parameter $\beta\in(-1,1)$. 
That is, the characteristic function of $S^{\alpha,\beta}_t$ is given by (see e.g.\cite{zolotarev1986one})
\begin{align*}
	\log \E \exp(\im \lambda S_t^{\alpha,\beta}) = -t|\lambda|^\alpha \left[ 1-\im \tan\left(\frac{\pi\alpha}{2}\right) \beta \sgn(\lambda) \right].
\end{align*}
The Lévy measure corresponding to this standardization can be expressed in the form \eqref{eqn:def-stablesum} with $M=1$, $\frac{r^+ - r^-}{r^+ + r^-}=\beta$, and $(r^+ + r^-) = \frac{1}{\Gamma(1-\alpha) \cos(\pi\alpha/2)}$ if $\alpha\neq 1$. 
Here, we will set $\beta=-1/3$ and study the cases $\alpha=1.3$ and $\alpha=1.7$.
Then \eqref{eqn:def-locstable} is satisfied with $\rho=0.5$, such that $S^{0.5,0}_t$ is a nuisance term, and $\tilde{Z}_t = B_t + S_t^{\alpha,\beta}$.
In view of applications in financial econometrics, we consider the time horizon $T=1$, and sampling frequencies $h=0.2/23400$,$h=1/23400$, and $5/23400$.
This sampling schemes correspond to $0.2$ resp.\ $1$ resp.\ $5$ seconds per quote on a trading day of $6.5$ hours.

To determine the solution of the estimating equation \eqref{eqn:def-GMM}, we need to compute the moments $\E_\theta \f(u\tilde{Z}_h)$ and their gradients.
This can be done numerically by means of a continuous Fourier transform since $\E\exp(\im\lambda\tilde{Z}_h)$ is available in closed form.
The employed moment functions $f_1,\ldots, f_4$ are handcrafted to satisfy \nameref{ass:F1} and \nameref{ass:F2}.
In our simulations, we use
\begin{align}
	\begin{split}
	f_1(x) &= 1-\exp(-10\,x^2),\\
	f_2(x) &= \exp\left( -\frac{300}{(|0.4 x| - 0.2)\vee 0} \right)\cdot \exp\left( -\frac{10}{(4-|0.4 x| )\vee 0} \right), \\
	f_3(x) &= \exp\left( -\frac{300}{(|1.6 x| - 0.2)\vee 0} \right)\cdot \exp\left( -\frac{10}{(4-|1.6 x| )\vee 0} \right), \\
	f_4(x) &= \begin{cases}
	 \exp\left( -\frac{300}{(|1.6 x| - 0.2)\vee 0} \right)\cdot \exp\left( -\frac{10}{(4-|1.6 x| )\vee 0} \right),& x\geq 0, \\
	  \exp\left( -\frac{300}{(|0.4 x| - 0.2)\vee 0} \right)\cdot \exp\left( -\frac{10}{(4-|0.4 x| )\vee 0} \right),& x<0.
	\end{cases}
	\end{split} \label{eqn:functions-MC}
\end{align}
Note that $f_2,f_3, f_4$ vanish on $[-1/8, 1/8]$.
We use the rescaling factor $u=1/\sqrt{h |\log h|}$. 
Although this choice of $u$ is too large to comply with assumption \nameref{ass:U}, we found it to perform better than smaller values for the given sampling scenario. 

The methods of \cite{reiss2013testing} and \cite{bull2016near} each have a tuning parameter $m\in\N$, and larger values of $m$ increase the rate of convergence. 
However, smaller values of $m$ can be superior in finite samples. 
In our simulations, we found that the estimator of Bull performed best when setting $m=3$, and the estimator of Reiß performed best when setting $m=2$, across all observation frequencies.
Furthermore, the method of Reiss involves a rescaling parameter $U_n$ and two weighting measures $w_1$, $w_2$. 
We choose the weighting measure $w_1$ to be supported on the set $\{1/m, 2/m, \ldots, 1\}$, and $w_2$ to be supported on the set $\{2/m, 4/m, \ldots, 2\}$. 
The truncation parameter is set to $U = h^{-(1-2m)/(4m-1)}$, as suggested by equation (3.8) therein.

In Table \ref{tab:errors}, we compare the simulated performance of our moment estimator for $\alpha$ and $\sigma^2$ with the estimators of \cite{jacod2014efficient}, \cite{reiss2013testing}, and \cite{bull2016near}.
For the latter two, we choose the best tuning parameter $m$ as specified above. 
The estimator of \cite{jacod2014efficient} is implemented as in equation (5.3) therein, with $\zeta=1.5$ and $u=|\log h|^\frac{1}{30}$. 
It is found that the new estimators perform best in the considered setting
The good performance of the estimator of Reiß in the case $\alpha=1.7$ is somewhat surprising, since the analysis of \cite{reiss2013testing} only yields a suboptimal rate of convergence.
However, for the latter estimator, no central limit theorem is available.
Hence, it is possible that the estimator in fact converges at a rate which is faster than the rate derived by \cite{reiss2013testing}.
It should also be noted that all benchmarked methods require various tuning parameters. 
Most notably, all methods require some form of scaling factors.
Furthermore, our new estimator depends on the the employed moment functions $f_j$, and the estimator of \cite{bull2016near} requires the choice of a truncation kernel function.
It is thus possible that a very careful choice of these parameters might affect the ranking implied by Table \ref{tab:errors}.

\begin{table}[tb]
\centering
	\begin{tabular}{lr|cc|ccc}
		  \toprule
		$\alpha$ & $h=n^{-1}$ & GMM $\hat{\sigma}^2$ & JT $\hat{\sigma}^2$ & GMM $\hat\alpha$ & Reiß $\hat\alpha$ & Bull $\hat\alpha$ \\ 
		  \midrule
			  1.3 & 5/23400   &  0.04 & 0.07 & 0.19 & 0.28 & 0.59 \\
			  1.3 & 1/23400   &  0.02 & 0.03 & 0.13 & 0.17 & 0.37 \\
			  1.3 & 0.2/23400 &  0.007 & 0.010 & 0.08 & 0.10 & 0.25 \\ \midrule
			  1.7 & 5/23400   &  0.32 & 0.43 & 0.23 & 0.22 & 0.31 \\
			  1.7 & 1/23400   &  0.16 & 0.22 & 0.11 & 0.11 & 0.30 \\
			  1.7 & 0.2/23400 &  0.08 & 0.10 & 0.06 & 0.06 & 0.25 \\
		   \bottomrule
	\end{tabular}
	\caption{Median absolute errors for the estimation of $\alpha$ and $\sigma^2$ in model \eqref{eqn:levy-sim}, for different estimators. All values are based on $20000$ simulations.}
	\label{tab:errors}
\end{table}

The volatility estimator $\hat{\sigma}^2$ is efficient, and from \eqref{eqn:sigma-efficient}, the error $\hat{\sigma}^2-\sigma^2$ should be of order $\sqrt{2\,h\sigma^4}$. 
From the results of Table \ref{tab:errors}, we find that this asymptotic performance is not achieved for the considered sample sizes. 
This defect holds for our proposed estimator as well as for the benchmark method of \cite{jacod2014efficient}, and it is bigger for large values of $\alpha$.
This is potentially due to the relatively large jump component of the simulated process \eqref{eqn:levy-sim}.
On the other hand, the asymptotic distribution of Theorem \ref{thm:estimating-clt} yields a good approximation of the finite sample behavior of $\hat{\alpha}$, as shown in Figure \ref{fig:alpha-hist}.
Clearly, the match with the asymptotic normal distribution improves for smaller $h$. 
Furthermore, the approximation is better for the smaller value $\alpha=1.3$.

\begin{figure}[tb!]
	\begin{subfigure}[c]{0.48\textwidth}
	\includegraphics[width=\textwidth]{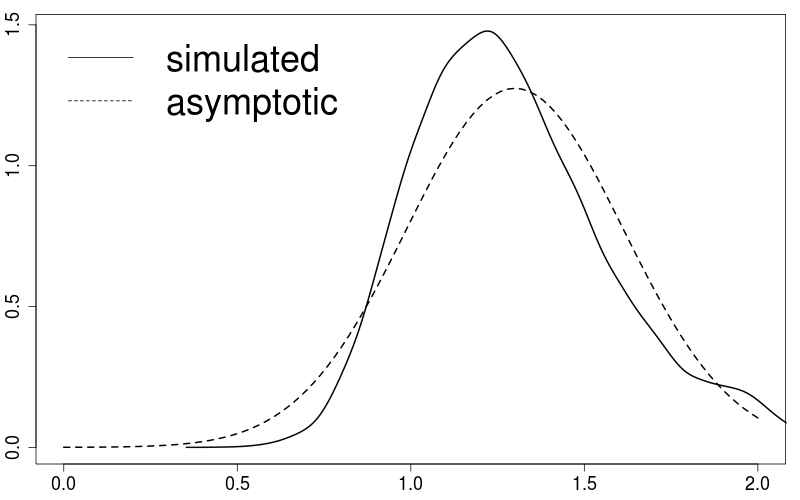}
	\subcaption{$h=5/23400$, $\alpha=1.3$}
	\end{subfigure}
	\begin{subfigure}[c]{0.48\textwidth}
	\includegraphics[width=\textwidth]{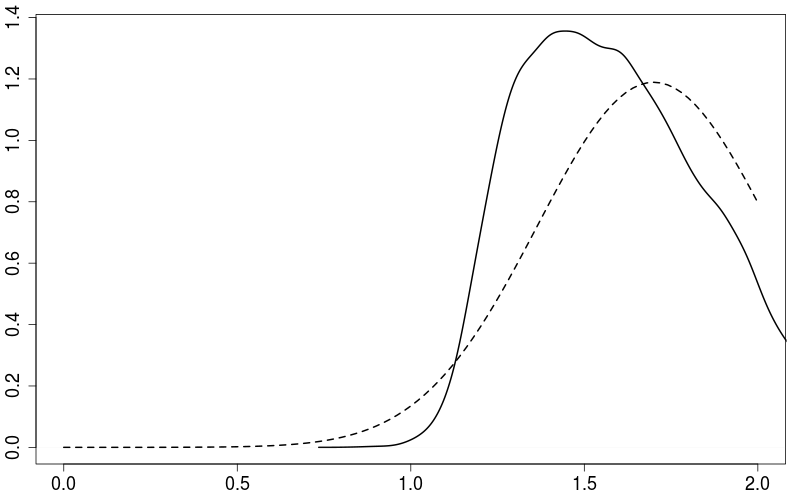}
	\subcaption{$h=5/23400$, $\alpha=1.7$}
	\end{subfigure}
	
	\begin{subfigure}[c]{0.48\textwidth}
	\includegraphics[width=\textwidth]{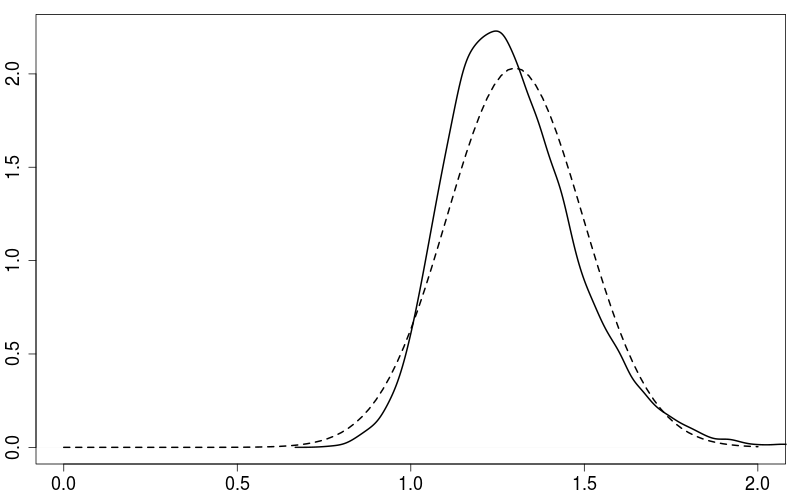}
	\subcaption{$h=1/23400$, $\alpha=1.3$}
	\end{subfigure}
	\begin{subfigure}[c]{0.48\textwidth}
	\includegraphics[width=\textwidth]{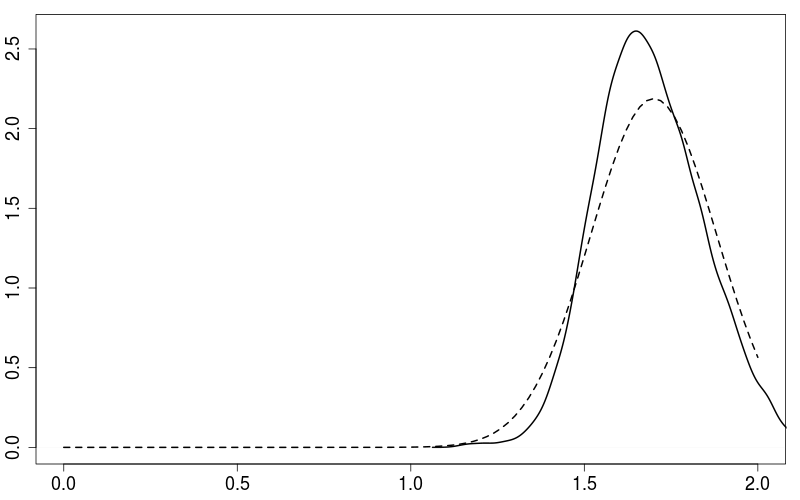}
	\subcaption{$h=1/23400$, $\alpha=1.7$}
	\end{subfigure}
	
	\begin{subfigure}[c]{0.48\textwidth}
	\includegraphics[width=\textwidth]{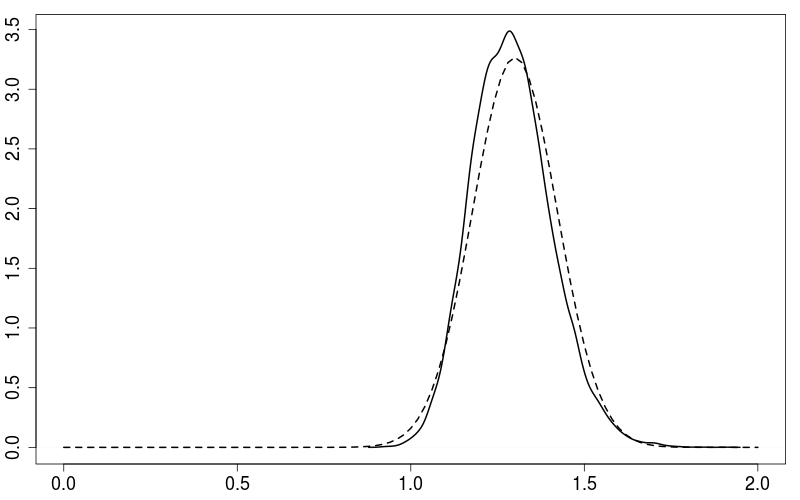}
	\subcaption{$h=0.2/23400$, $\alpha=1.3$}
	\end{subfigure}	
	\begin{subfigure}[c]{0.48\textwidth}
	\includegraphics[width=\textwidth]{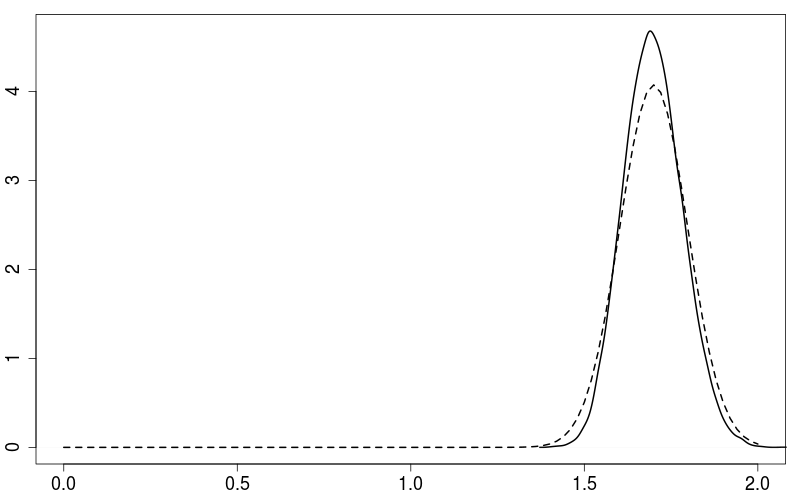}
	\subcaption{$h=0.2/23400$, $\alpha=1.7$}
	\end{subfigure}
\caption{Simulated and asymptotic distributions of the proposed estimator $\hat{\alpha}$, based on $20000$ simulations.}
\label{fig:alpha-hist}
\end{figure}

\section{Technical tools}\label{sec:technical}

In this section, we present the proofs of Theorem \ref{thm:estimating-clt} and Propositions \ref{prop:fisher} and \ref{prop:clt-single}. 
Preliminary technical results are presented in Subsection \ref{sec:technical-levy}, as they might be of independent interest, in particular Lemma \ref{lem:moments} and Corollary \ref{cor:bias-complete}.
The detailed proofs are presented in Subsection \ref{sec:proofs-levy}.

\subsection{Preliminary results}\label{sec:technical-levy}


To study the asymptotic behavior of the estimating equation \eqref{eqn:def-GMM} by standard techniques (see e.g.\ \cite{jacod2017review}), we need \begin{itemize}
	\item a central limit theorem for the term $\frac{1}{n} \sum_{i=1}^n \f (u_n \Delta_{n,i}X) - \E_\theta \f(u_n\tilde{Z}_h)$, and
	\item properties of the derivatives $\D_\theta \E_\theta \f(u_n\tilde{Z}_h)$.
\end{itemize} 

To determine asymptotic variances, as well as for some technical steps of the following proofs, it is useful to derive some explicit approximations of $\E \f(u_n \tilde{Z}_h)$.

\begin{lemma}\label{lem:moments}
	Let $f\in\mathcal{C}^2$ be such that $f, f'$ and $f''$ are bounded and $f(0)=0$, and let $\tilde{X}_t$ be a Lévy process with characteristic triplet $(\mu,\sigma^2, \tilde{\nu})$. 
	The implicit constants in the following expressions depend on $f$ and $(\mu,\sigma^2, \tilde{\nu})$, but neither on $t$ nor on $u$.
	Moreover, all $\mathcal{O}(\cdot)$ and $o(\cdot)$ terms are bounded resp.\ vanishing uniformly on compacts in $\Theta$.
	\begin{enumerate}[(i)]
		\item If $f(x)=0$ for $|x|\leq \eta$, then for any $\lambda\in(0,1)$ such that $u \leq \frac{(1-\lambda)\eta}{\sigma \sqrt{8 t |\log t|}}$, as $t\to 0$, \begin{align*}
				\E f(u\tilde{X}_t) =o(tu^\alpha ) +  t u^{\alpha} \left[ r_1^+ \tilde{\mathcal{J}}_{\alpha}^+f(0) + r_1^- \tilde{\mathcal{J}}_{\alpha}^-f(0)\right] ,
			\end{align*}
			where 
			\begin{align*}
			\tilde{\mathcal{J}}_{\alpha}^\pm f(x) &= \int \alpha \mathds{1}_{\pm z >0} \frac{\left[f(x+z) - f(x) - f'(x)\mathds{1}_{|z|\leq 1} \right]}{|z|^{1+\alpha}}\,\tilde{\nu}(dz).
			\end{align*}
		\item If, alternatively, $f(0)=0$ but $f''(0)\neq 0$, then for any $u=o(1/\sqrt{t})$\begin{align*}
			\E f(u\tilde{X}_t) = tu^2\frac{\sigma^2}{2}f''(0) + o(tu^2).
		\end{align*}
		\item If $f(0)=0, f''(0)=0$ but $f^{(4)}\neq 0$, and $f^{(3)}, f^{(4)}$ are bounded, then for any $u=o(1/\sqrt{t})$ \begin{align*}
			\E f(u\tilde{X}_t) = t^2u^4\frac{\sigma^4}{8}f^{(4)}(0) + o(t^2u^4) + \mathcal{O}(tu^\alpha).
		\end{align*}
		\item If $f(0)=0$ and $\mu=0$, $\sigma^2=0$, then there exists a constant $\tilde{C}$ bounded uniformly on compacts, such that for all $f$ and all $u>1$, $t\geq 0$, \begin{align*}
			\E f(u\tilde{X}_t) \leq t u^{\alpha\vee 1} (1+\log(u)) \left( \|f\|_\infty + \|f'\|_\infty + \|f''\|_\infty \right) \tilde{C}.
		\end{align*}
	\end{enumerate}
\end{lemma}

The case (i), which is exploited in the proofs several times, imposes a subtle upper bound on $u$. 
Although this bound need not be sharp, the Lemma will not hold for $u = \tau/\sqrt{t|\log t|}$ if $\tau$ is too large.
To make this plausible, note that for an $\alpha$-stable process $S_t^\alpha$, the probability $P(|S_t^\alpha| \geq \eta \sqrt{t|\log t|}/\tau)$ tends to zero as $t\to 0$, roughly polynomially in $t$.
On the other hand, for the Brownian motion, $P(|B_t|>\eta \sqrt{t|\log t|} / \tau) = P(|B_1|> \eta \sqrt{|\log t|}/\tau)\to 0$ polynomially as well, but the polynomial order of this decay will depend on the specific value of $\tau$.
For the jump term to dominate, as in case (i) of Lemma \ref{lem:moments}, $\tau$ must be small.
The uniformity w.r.t.\ $\theta$ of the previous results will be used later on to derive the consistency of the estimator.

Another ingredient to obtain a central limit theorem is a bias bound, i.e.\ a bound on the error of approximating $\E\f(u_n\Delta_{n,i} X)$ by $\E_\theta \f(u_n\tilde{Z}_h)$.
For two random variables $X$ and $Y$, recall the definition of the 1-Wasserstein metric $d_W$ and the total variation distance $d_{TV}$ given by \begin{align*}
	d_{TV}(X,Y) &= \sup_{g: \|g\|_\infty\leq 1} \left| \E g(X) - \E g(Y) \right|, \\
	d_{W}(X,Y)  &= \sup_{g: \|g'\|_\infty\leq 1} \left| \E g(X) - \E g(Y) \right|,
\end{align*}
where the supremum is taken over all bounded resp.\ Lipschitz continuous, measurable functions $g:\R\to\R$.
These distances are used in the proof of the following Lemma, which quantifies the error of approximation implied by the local stability assumption \eqref{eqn:def-locstable}. 

\begin{lemma}\label{lem:bias-bound}
	Let $X_t, \tilde{X}_t$ be two Lévy processes with characteristic triplets $(\mu,\sigma^2,\nu)$ and $(\mu, \sigma^2, \tilde{\nu})$, respectively. 
	Suppose furthermore that for some $\rho\in(0,1\wedge \alpha)$, 
	\begin{align*}
		\left| \nu((z,\infty)) - \tilde{\nu}((z,\infty)) \right| &\leq L |z|^{-\rho}, \quad z\in(0,1),\\
		\left| \nu((-\infty,z)) - \tilde{\nu}((-\infty,z)) \right| &\leq L |z|^{-\rho}, \quad z\in(-1,0).
	\end{align*}
	There exists a constant $\tilde{C}$ depending on $L$, $\rho$, and $\theta$, such that for any differentiable function $f:\R\to\R$, and any $u> 1$,
	\begin{align}
		\label{eqn:fu-firstbound}\left| \E f(uX_t) - \E f(u\tilde{X}_t - ut\bar\zeta) \right| &\leq \tilde{C} (\|f\|_\infty +\|f'\|_\infty +\|f'\|_{L_1})(tu^{\rho}+t^2u^{\alpha+1}),
	\end{align}
	where $\bar\zeta = \int \trunc(z) (\nu-\tilde{\nu})(dz) \in \R$. 
	The constant $\tilde{C}$ is bounded on compacts in $\theta\in\Theta$, $\rho\in(0,1\wedge \alpha)$, and $L\geq 0$. 
\end{lemma}

\begin{cor}\label{cor:bias-complete}
	Let $f\in \mathcal{C}^3$ such that $f, f', f'', f'''$ are bounded and $f'\in L_1$.
	Let $X_t, \tilde{X}_t$ be two Lévy processes with characteristic triplets $(\mu,\sigma^2 ,\nu)$ and $(0,\sigma^2,\tilde{\nu})$, respectively. 
	Suppose that $\nu$, $\tilde{\nu}$ satisfy the conditions of Lemma \ref{lem:bias-bound}. 
	Then, as $t\to 0$, 
	\begin{align}
		\label{eqn:bias-withdrift}|\E f(uX_t) - \E f(u\tilde{X}_t)| &\leq \tilde{C}\left(tu^{\rho} + t^2u^{2\vee(\alpha+1)}\right)(1+\log(u))).
	\end{align}
	The constant $\tilde{C}$ is bounded on compacts in $\mu\in\R$, $\theta\in\Theta$, $\rho\in(0,1\wedge \alpha)$, and $L\geq 0$.
\end{cor}

Note that the presented result of \ref{cor:bias-complete} can not be directly formulated in terms of $d_{TV}$ or $d_{W}$, distinguishing it from the results of \cite{mariucci2017wasserstein}.
An alternative bound on the total variation distance between $X_t$ and $\tilde{Z}_t$ is presented by \cite[Proposition 4]{Clement2018} and \cite[Proposition 2]{Amorino2019}, stating that $d_{TV}(X_t, \tilde{Z}_t) \leq C t^{1 \wedge \frac{1}{\alpha}} \log(t)$ as $t\to 0$.
Their assumptions on the Lévy measure $\nu(dz)$ imply that our condition \eqref{eqn:def-locstable} holds, with $\rho\leq (\alpha-1)\vee 0$.
Thus, if $\alpha>1$ and $u\ll t^{-1/2}$, our bound \eqref{eqn:bias-withdrift} is sharper since $tu^{\alpha-1}\ll t^{\frac{3}{2}-\frac{\alpha}{2}} \ll t^{\frac{1}{\alpha}}$.
In the case $\alpha\leq 1$, our bound is of the same order of magnitude as the one presented by \cite{Clement2018} and \cite{Amorino2019}.
Furthermore, our result may also be applied in the case $\rho>\alpha-1$.
However, we impose additional smoothness assumptions upon the considered function $f$, which is suitable for our statistical purposes because the moment functions are chosen by the statistician.

To state the remaining technical results, introduce the notation \begin{align*}
	\Lambda_n(\theta) 
		&= \diag(hu^2, hu^{\alpha_1},hu^{\alpha_1},hu^{\alpha_1},\ldots \\
		&\qquad\qquad \ldots,	hu^{\alpha_M},hu^{\alpha_M},hu^{\alpha_M})\in \R^{(3M+1) \times (3M+1)},\\
	\tilde{\Lambda}_n(\theta) 
		&= \diag(hu^2, \sqrt{hu^{\alpha_1}},\ldots,\sqrt{hu^{\alpha_1}})\in \R^{(3M+1) \times (3M+1)} ,
\end{align*}
such that 
\begin{align*}
	\bar{\Lambda}_n(\theta) 
	&= \tilde{\Lambda}_n^{-1}(\theta)\Lambda_n(\theta) \\
	&= \sqrt{h}\,\diag(\sqrt{h}^{-1},  u^{\alpha_1 - \frac{\alpha_1}{2}}, u^{\alpha_1 - \frac{\alpha_1}{2}},  u^{\alpha_1 - \frac{\alpha_1}{2}},\ldots \\
	&\qquad\qquad \ldots,  u^{\alpha_M - \frac{\alpha_1}{2}}, u^{\alpha_M - \frac{\alpha_1}{2}}, u^{\alpha_M - \frac{\alpha_1}{2}}).
\end{align*}

Corollary \ref{cor:bias-complete} and Lemma \ref{lem:moments} allow us to derive the following central limit theorem for the estimated moments.
In particular, we use Lemma \ref{lem:moments} to control the sampling variance, and Corollary \ref{cor:bias-complete} to control the bias.

\begin{lemma}\label{lem:moment-clt}
	Let $n h_n = T=1$ constant, i.e.\ $h_n= 1/n$, and choose $u_n\to\infty$ according to \nameref{ass:U}.
	Let $\f$ satisfy \nameref{ass:F1} and \nameref{ass:F2}, and suppose that the Lévy process $X_t$ satisfies \eqref{eqn:def-locstable} with some $\rho<\alpha/2$.  
	Then, as $n\to\infty$,
	\begin{align*}
		\tilde{\Lambda}_n^{-1}(\theta)\frac{1}{\sqrt{n}}\sum_{i=1}^n \f(u_n\Delta_{n,i} X) - \E_\theta\f(u_n\tilde{Z}_h)   \quad\wconv\quad \mathcal{N}\left(0,\Sigma(\theta)\right).
	\end{align*}
\end{lemma}

Note that the rate of convergence for the first moment $f_1$ is slower than for $f_j, j\geq 2$. 
This is due to our special choice of $f_j, j\geq 2$, which vanish near zero. 
Hence, these moments are primarily driven by the jump component, which is of a smaller order than the diffusion term.
On the other hand, the jump parameters $\alpha_m, r_m^\pm$ are harder to identify, i.e.\ $\partial_{\alpha_m} \E_\theta \f(u \tilde{Z}_h) \ll \partial_{\sigma^2} \E_\theta \f(u \tilde{Z}_h)$.
This is established in the following Lemma.

\begin{lemma}\label{lem:expect-derivative}
	Let $f\in\mathcal{C}^2(\R)$ be such that $f, f', f''$ are bounded.
	Let $\tilde{X}_t$ be a Lévy process with characteristic triplet $(0,\sigma^2, \tilde{\nu})$, parameterized by $\theta$ as in \eqref{eqn:parameter-vector}. Then, as $h\to 0$, $u\to\infty$, such that $hu^2\to 0$, 
	\begin{align}
		\label{eqn:expect-deriv-result-1}
		\begin{split}
				\partial_{\sigma^2} \E_\theta f(u\tilde{X}_t) 
			&= h \frac{u^2}{2} f''(0) + o(hu^2) , \\
				\partial_{r_m^\pm} \E_\theta f(u\tilde{X}_h) 
			&= h u^{\alpha_m}  \mathcal{J}_{\alpha_m}^\pm f(0) + o(hu^{\alpha_m}) + \mathcal{O}\left(hu^{\alpha_m\vee 1} \log u\right) \E_\theta f'(u\tilde{X}_h),\\
				\partial_{\alpha_m} \E_\theta f(u\tilde{X}_h) 
			&=  hu^{\alpha_m} (\log u) \left[ r_m^+\mathcal{J}_{\alpha_m}^+f(0) +r_m^-\mathcal{J}_{\alpha_m}^-f(0) \right] \\
			&\quad + o(hu^{\alpha_m} \log u) + \mathcal{O}\left(hu^{\alpha_m\vee 1} (\log u)^2\right) \E_\theta f'(u\tilde{X}_h),
		\end{split}
	\end{align}
	and, 
	\begin{align}
		&\quad \left(\partial_{\alpha_m} - \log(u)\left( r_m^+ \partial_{r_m^+} + r_m^- \partial_{r_m^-} \right)\right) \E_\theta f(u\tilde{X}_h) \nonumber \\
		\begin{split}
			&=hu^{\alpha_m} \partial_{\alpha_m} \left[r_m^+\mathcal{J}_{\alpha_m}^+ f(0) + r_m^-\mathcal{J}_{\alpha_m}^- f(0)\right] + o(hu^{\alpha_m}) \\
			&\qquad + \mathcal{O}\left(hu^{\alpha_m\vee 1} (\log u)^2\right) \E_\theta f'(u\tilde{X}_h).
		\end{split}\label{eqn:expect-deriv-result-2}
	\end{align}
	Moreover, if $f$ vanishes on $[-\eta,\eta]$ and $u$ satisfies Condition \nameref{ass:U}, \begin{align}
		\partial_{\sigma^2} \E_\theta f(0) = o(hu^\alpha).\label{eqn:expect-deriv-result-3}
	\end{align}
	All terms of the form $\mathcal{O}(\cdot)$ and $o(\cdot)$ are bounded resp.\ vanishing uniformly on compacts in $\Theta$.
\end{lemma}

\begin{cor}\label{cor:moment-derivative}
	Let $\f$ satisfy \nameref{ass:F1} and \nameref{ass:F2}, and let $\tilde{X}_{t}$ be a Lévy process with characteristic triplet $(0,\sigma,\tilde{\nu})$, parameterized by $\theta$ as in \eqref{eqn:parameter-vector}. Then, as $h=\frac{1}{n}\to 0$, $u_n\to\infty$, such that $u_n = o(\sqrt{h})$, 
	\begin{align}
	\begin{split}
		\tilde{\Lambda}_n^{-1}(\theta) \left[\D_\theta \E_\theta \f(u_n\tilde{X}_h)\right] \Gamma_n(\theta) \bar{\Lambda}_n^{-1}(\theta) &\to A(\theta).
	\end{split} \label{eqn:deriv-conv}
	\end{align}
	This convergence holds uniformly on compacts in $\theta\in\Theta$.
\end{cor}

These results allow us to establish the consistency of $\hat{\theta}_n$. We do not consider global uniqueness of the solution of the estimating equation \eqref{eqn:def-GMM}. Hence, we only obtain the existence of a consistent sequences of random variables satisfying the equation.

\begin{lemma}[Consistency]\label{lem:estimating-consistency} 
	Let $X_t$ be a Lévy process satisfying \eqref{eqn:def-locstable} with some $\rho<\alpha/2$, and parameter vector $\theta_0$. 
	Let $\f$ satisfy assumptions \nameref{ass:F1}, \nameref{ass:F2}, and \nameref{ass:I}, and let $u_n\to\infty$ be chosen according to \nameref{ass:U}.
	There exists a sequence of random vectors $\hat{\theta}_n$ solving \eqref{eqn:def-GMM}, such that $\hat{\theta}_n\to\theta$ in probability as $n\to\infty$. 
	This sequence is eventually unique, i.e.\ for any other consistent sequence $\hat{\theta}_n^*$ solving the estimating equation, it holds $P(\hat{\theta}_n \neq \hat{\theta}_n^*)\to 0$.
\end{lemma}

To obtain a central limit theorem for $\hat{\theta}_n$, we may apply a Taylor expansion to obtain the representation 
\begin{align*}
	\hat{\theta}_n -\theta_0 \approx -\left[\widetilde{\D_\theta\f}\right]^{-1}  \frac{1}{n}\left[\sum_{i=1}^n \f(u_n\Delta_{n,i}X) - \E_{\theta_0}\f(u_n\tilde{Z}_h)  \right],
\end{align*}
where $\widetilde{\D \f}_{j,k} = \partial_{\theta_k} \E_{\tilde{\theta}^j} f_j(u_n \tilde{Z}_h)$ for some $\tilde{\theta}^j$ on the line segment between $\theta_0$ and $\hat{\theta}_n$, for $j=1,\ldots, 3M+1$. 
This standard approach allows to establish Theorem \ref{thm:estimating-clt}, as detailed in Subsection \ref{sec:proofs-levy}.

\subsection{Proofs}\label{sec:proofs-levy}

\begin{proof}[Proof of Lemma \ref{lem:moments}] 
	At the price of changing the term $\mu$, we may assume w.l.o.g.\ that $\trunc(z) = z\mathds{1}_{|z|\leq 1}$. In view of the Lévy-Itô decomposition \eqref{eqn:levy-ito}, we write 
	\begin{align*}
		u \tilde{X}_t &= u \mu t + u \sigma B_t + \int uz \left(N(dz,ds) - \mathds{1}_{|z|\leq \frac{1}{u}}\tilde{\nu}(dz)\otimes ds \right)\\
		&\quad + t \int uz (\mathds{1}_{|z|\leq \frac{1}{u}} - \mathds{1}_{|z|\leq 1} ) \tilde{\nu}(dz) \\
		&= u \mu t + u \sigma B_t + J_t^u + u t \mu_u
	\end{align*}
	where $N$ is a Poisson counting measure with intensity $\tilde{\nu}(dz)\otimes ds$, and $J_t^u$ denotes the corresponding integral term. 
	The explicit form of $\tilde{\nu}$ allows for computation of $\mu_u$, as
	\begin{align*}
		|\mu_u| 
		& \leq \int_{\frac{1}{u}}^{1} \sum_{m=1}^M (r_m^++r_m^-)|z|^{-\alpha_m}\, dz\\
		& \leq \sum_{m=1}^M (r_m^++r_m^-) (u^{\alpha_1-1}+1)(\alpha_1^{-1}+\log(u))) \quad 
		\leq u^{(\alpha_1-1)\vee 0}(1+\log(u))\sum_{m=1}^M (r_m^++r_m^-).
	\end{align*}
	The term $\log(u)$ is added to cover the case $\alpha_1=1$.
	This bound on $\mu_u$ will be used in the sequel.
	
	To derive the claims of the Lemma, we start with a rough bound for the probability
	\begin{align}
		\begin{split}
		P(|u\tilde{X}_t|>\eta) &\leq P\left( |u t (\mu+\mu_u) t| >\frac{1-\lambda}{2}\eta \right)\\
		&\quad + P\left(|\sigma uB_t|>\frac{1-\lambda}{2}\eta\right) +  P \left( |J_t^u| > \lambda \eta \right), \quad\lambda\in(0,1). 
		\end{split} \label{eqn:lem-moments-prob}
	\end{align}	
	The first term tends to zero identically as $t\to 0$. 
	To study the jump term, choose a bounded, smooth function $ g(x)\geq  \mathds{1}_{|x|\geq \lambda \eta }$ such that $g(0)=g'(0)=0$. 
	Then by Itô's formula, and a substitution in the integral, we obtain \begin{align*}
		P\left( |J_t^u| > \lambda \eta \right) &\leq \E g(J_t^u) \\
		&= \int_0^t \int \E\left[g(J_s^u+uz) - g(J_s^u) - g'(J_s^u) uz \mathds{1}_{|z| \leq \frac{1}{u}} \right]\, \tilde{\nu}(dz)\,  ds \\
		&\leq \sum_{m=1}^M (r_m^++r_m^-) \alpha_m u^{\alpha_m} \int_0^t \int \frac{\E\left[g(J_s^u+z) - g(J_s^u) - g'(J_s^u) z \mathds{1}_{|z| \leq 1} \right]}{|z|^{1+\alpha_m}}\, \tilde{\nu}(dz)\,  ds \\
		&\leq \tilde{C} u^\alpha (\|g\|_\infty + \|g''\|_\infty),
	\end{align*}
	for a constant $\tilde{C}$ depending on $\boldsymbol{\alpha}, \boldsymbol{r}$ and is bounded on compacts in these parameters.
	The function $g$ can be chosen such that the latter term is finite. 
	Thus, $P(|u\tilde{X}_t| > \lambda\eta) = \mathcal{O}(u^\alpha t)$, uniformly on compacts in $\boldsymbol{\alpha}, \boldsymbol{r}$. 
	
	For the Gaussian term in \eqref{eqn:lem-moments-prob}, we employ the tail bound  \begin{align*}
		P\left( |B_1| > \frac{(1-\lambda)\eta}{2\sigma u \sqrt{t}} \right) \;\leq\; \frac{2\sigma u \sqrt{t}}{(1-\lambda)\eta \sqrt{2\pi}}\exp\left(\frac{-\eta^2(1-\lambda)^2}{8 \sigma^2 u^2 t} \right).
	\end{align*} 
	Now let $a>0$ be such that $u = \frac{(1-\lambda)\eta}{\sqrt{a}\sigma \sqrt{8t|\log t|}}$. 
	Then 
	\begin{align*}
		P\left( |B_1| > \frac{(1-\lambda)\eta}{2\sigma u \sqrt{t}} \right) \;\leq\; \frac{\exp(a \log t )}{2\sqrt{ a \pi |\log t|}}  = \frac{t^a}{2\sqrt{ a \pi |\log t|}}.
	\end{align*}
	If $a\geq 1$, i.e.\ $u\leq \frac{(1-\lambda)\eta}{\sigma\sqrt{8t|\log t|}}$, the latter bound is of order less than $\mathcal{O}(u^\alpha t)$, uniformly on compacts. In particular, \begin{align*}
		P(|u\tilde{X}_t| > \eta) \leq \tilde{C} t u^\alpha.
	\end{align*}
	Note that the latter inequality does not hold if $u=\tau/\sqrt{-t\log t}$ for a proportionality factor $\tau$ which is too large. 
	
	If $u$ is larger, but $u=o(1/\sqrt{t})$,the bound on $P(|J_t^u|>\lambda\eta)$ remains unchanged, while we still obtain $P(|u\sigma B_t|>\eta)\to 0$ uniformly on compacts.
	Thus, if we only suppose $u=o(1/\sqrt{t})$, we have $P(|u\tilde{X}_t|>\tilde{\eta}) \to 0$ uniformly on compacts, for any $\tilde{\eta}>0$, but with a slower rate.
	
	To obtain an asymptotically exact value, we plug the former rough bound into It{\^o}'s formula. 
	In case (i), we have 
	\begin{align}
		\E f(u\tilde{X}_t) &= \E\int_0^t\Big[ \frac{u^2\sigma^2}{2} f''(u\tilde{X}_s) + (\mu+\mu_u) uf'(u\tilde{X}_s) \nonumber \\
		&\qquad + \int  (f(u\tilde{X}_s + uz)-f(u\tilde{X}_s) - uz\mathds{1}_{|uz|\leq 1} f'(u\tilde{X}_s))\,\tilde{\nu}(dz) \Big] ds \nonumber \\
		\begin{split}
		&= \int_0^t \frac{u^2\sigma^2}{2} \E f''(u\tilde{X}_s) + (\mu+\mu_u) u\E f'(u\tilde{X}_s) \,ds \\
		&\quad + \sum_{m=1}^M u^{\alpha_m} \int_0^t\left[ r_m^+\E\mathcal{J}_{\alpha_m}^+ f(u\tilde{X}_s) + r_m^- \E\mathcal{J}_{\alpha_m}^- f(u\tilde{X}_s) \right]\, ds
		\end{split} \label{eqn:JA-approx-Ito} \\
		&= u^2t \, \mathcal{O}(u^\alpha t )  + \sum_{m=1}^M u^{\alpha_m} \int_0^t\left[ r_m^+\E\mathcal{J}_{\alpha_m}^+ f(u\tilde{X}_s) + r_m^- \E\mathcal{J}_{\alpha_m}^- f(u\tilde{X}_s) \right]\, ds. \nonumber
	\end{align}
	Here, we used $\E f''(u\tilde{X}_s) \leq \|f''\|_\infty P(|u\tilde{X}_s| > \eta) = \mathcal{O}(u^\alpha t)$ as $f$ vanishes on $[-\eta,\eta]$. 
	We moreover used that $\E f'(u\tilde{X}_s) = \mathcal{O}(u^\alpha t)$, and $\mu_u u = \mathcal{O}(u^2t)$ as established previously. 
	These upper bounds hold uniformly on compacts in $\Theta$.
	To proceed, note that $\mathcal{J}_\alpha^\pm f$ is a bounded continuous function, since 
	\begin{align*}
		|\mathcal{J}_{\alpha}^\pm f(x)| \leq 2\|f\|_\infty \int_{|z|\geq 1} \frac{\alpha}{|z|^{1+\alpha}}dz + \|f''\|_\infty \int_{|z|\leq 2} \frac{\alpha|z|^2}{|z|^{1+\alpha}} dz,
	\end{align*} 
	which is furthermore bounded uniformly on compacts in $\alpha$. 
	By virtue of this boundedness, $u\tilde{X}_s\pconv 0$ implies $\E\mathcal{J}_{\alpha_m}^\pm f(u\tilde{X}_s) = \mathcal{J}_{\alpha_m}^\pm f(0)+o(1)$. 
	To ensure that this last approximation holds uniformly on compacts in $\Theta$, note that 
		$\|(\mathcal{J}_{\alpha_m}^\pm f)'\|_\infty = \|\mathcal{J}_{\alpha_m}^\pm f'\|_\infty$ 
	is also bounded, such that it suffices to control $\E( |u\tilde{X}_s|\wedge 1)$ uniformly. 
	But we already established that for any $\eta$, $P(|u\tilde{X}_s|>\eta)\to 0$ uniformly on compacts in $\Theta$. 
	Hence,
	\begin{align*}
		\E f(u\tilde{X}_t)
		&= u^2t \mathcal{O}(u^\alpha t ) + \sum_{m=1}^M u^{\alpha_m} \int_0^t\left[ r_m^+\E\mathcal{J}_{\alpha_m}^+ f(u\tilde{X}_s) + r_m^- \E\mathcal{J}_{\alpha_m}^- f(u\tilde{X}_s) \right]\, ds \\
		&= o(u^\alpha t) + \sum_{m=1}^M u^{\alpha_m} (r_m^+ \mathcal{J}_{\alpha_m}^+ f(0) + r_m^- \mathcal{J}_{\alpha_m}^- f(0))\\
		&= o(u^\alpha t) + u^\alpha t \left[r_1^+ \mathcal{J}_\alpha^+ f(0) + r_1^- \mathcal{J}_\alpha^- f(0)\right],
	\end{align*}
	uniformly on compacts in $\sigma^2,\boldsymbol{\alpha}, \boldsymbol{r}$.
	This proves the first claim. 
	
	If, on the other hand, $f(0)=0, f''(0)\neq 0$, a different term dominates in \eqref{eqn:JA-approx-Ito}. 
	We obtain 
	\begin{align*}
		\E f(u\tilde{X}_t)&=\int_0^t \frac{u^2\sigma^2}{2} \E f''(u\tilde{X}_s)\, ds + \mathcal{O}(u^\alpha t)\\
		&=\mathcal{O}(tu^\alpha) + \frac{u^2t}{2}\left(f''(0)+o(1)\right),
	\end{align*}
	uniformly on compacts in $\Theta$.
	
	For the case $f''(0)=0, f^{(4)}(0)\neq 0$, we may apply the result of case (ii) to obtain 
		$\E f''(u\tilde{X}_t) = \frac{u^2t\sigma^2}{2}f^{(4)}(0) + o(u^2t)$, 
	and hence \begin{align*}
		\E f(u\tilde{X}_t)
		&=\int_0^t \frac{u^2\sigma^2}{2} \E f''(u\tilde{X}_s) ds+ \mathcal{O}(u^\alpha t) \\
		&= \int_0^t \frac{u^4\sigma^4}{4}s f^{(4)}(0) ds+ \mathcal{O}(u^\alpha t) + o(u^4t^2) \\
		&=  \frac{u^4t^2\sigma^4}{8} f^{(4)}(0) ds+ \mathcal{O}(u^\alpha t) + o(u^4t^2).
	\end{align*}
	
	For the last claim, we use Itô's formula again. 
	Recall that the truncation function satisfies $\trunc(z)=z$ for $|z|\leq 1$, and $|\trunc(z)|\leq 2$.
	Then
	\begin{align*}
		\E f(u\tilde{X}_t) 
		&= \E\int_0^t  \int \left[ f(u(\tilde{X}_s+z)) - f(u\tilde{X}_s) - uf'(u\tilde{X}_s) \trunc(z) \right] \, \tilde{\nu}(dz) \\
		&\leq 2\,t \|f\|_\infty \tilde{\nu}\left(\left(-\frac{1}{u},\frac{1}{u}\right)^c\right) + 2\,tu\|f'\|_\infty \tilde{\nu}((-1,1)^c) \\
		&\quad + t u\|f'\|_\infty \int_{(-1,1)\setminus (-\frac{1}{u}, \frac{1}{u})} |z|\, \tilde{\nu}(dz) + t u^2 \|f''\|_\infty \int_{-\frac{1}{u}}^{\frac{1}{u}} z^2\, \tilde{\nu}(dz) \\
		&\leq t \tilde{C}\left( \|f\|_\infty u^\alpha + u \|f'\|_\infty + u \|f'\|_\infty (u^{\alpha-1}+1) + u^2 \|f''\|_\infty u^{\alpha-2} \right) (1+\log(u))\\
		&\leq t \tilde{C} u^{\alpha\vee 1} (1+\log(u)) \left( \|f\|_\infty + \|f'\|_\infty + \|f''\|_\infty \right)
	\end{align*}
	The additional factor $\log(u)$ is introduced to cover the special case $\alpha=1$ when computing the integral $\int_{1/u}^1 |z|^{-\alpha} dz$.
\end{proof}

\begin{proof}[Proof of Lemma \ref{lem:bias-bound}] 
Choose some $0<\epsilon<\frac{1}{u}$. The process $X_t$ may be decomposed by virtue of the Lévy-Itô decomposition as \begin{align*}
	X_t &= \mu t + \sigma B_t + \int_0^t\int (z-\trunc(z)) \, N(dz, ds) + \int_0^t\int \trunc(z) \, (N-\nu)(dz,ds)\\
	 &= \mu t + \sigma B_t + J^1_t + J^2_t +J^3_t - t \zeta_\epsilon, \\
	J^1_t &= \int_0^t \int_{[-\epsilon,\epsilon]} \trunc(z) (N-\nu)(dz,dt), \\
	J^2_t &= \sum_{s\leq t} \Delta X_s \mathds{1}_{\epsilon <|\Delta X_s| \leq \frac{1}{u}}, \\
	J^3_t &= \sum_{s\leq t} \Delta X_s \mathds{1}_{\frac{1}{u}<|\Delta X_s|} ,\\
	\zeta_\epsilon &= \int_{|z|>\epsilon} \trunc(z)\, \nu(dz),
\end{align*}
where $(N-\nu)$ is a compensated homogeneous Poisson point process with intensity measure $\nu(dz)$, such that $J_t^1$ is a martingale. 
For $\tilde{X}_t$, we have the analogous decomposition $\tilde{X}_t= \mu t + \sigma B_t + \tilde{J}^1_t + \tilde{J}^2_t + \tilde{J}^3_t + t \tilde{\zeta}_\epsilon$. 
Moreover, 
\begin{align*}
	\zeta_\epsilon - \tilde{\zeta}_\epsilon &= \int_{|z|>\epsilon} \trunc(z)\,(\nu-\tilde{\nu})(dz)\\
	&= \int_{\epsilon<|z|<1} z\, (\nu-\tilde{\nu})(dz) + \int_{|z|>1 } \trunc(z)\, (\nu-\tilde{\nu})(dz) .
\end{align*}
The second integral is finite. 
Furthermore, integrating by parts, 
\begin{align*}
		\int_{\epsilon}^{1} z (\nu-\tilde{\nu})(dz)
	&=\int_{\epsilon}^{1} \left[\nu((z,1]) - \tilde{\nu}((z,1]) \right]dz + \epsilon \left[\nu((\epsilon,1]) - \tilde{\nu}((\epsilon,1]) \right],
\end{align*}
which has a limit as $\epsilon\to 0$ if $\rho<1$. 
Thus, there exists a real number $\bar\zeta$ such that $\zeta_\epsilon-\tilde{\zeta}_\epsilon \to \bar{\zeta}$ as $\epsilon\to 0$.

By subadditivity of the total variation distance and the Wasserstein distance, 
\begin{align}
	\nonumber&\quad \left| \E f(uX_t) - \E f(u\tilde{X}_t - ut \bar\zeta) \right| \\
	\nonumber &\leq \left| \E f\left(uX_t\right) - \E f\left(u ((\mu-\tilde{\zeta}_\epsilon - \bar\zeta) t+\sigma B_t + \tilde{J}^1_t + \tilde{J}^2_t + J_t^3) \right) \right| \\
	\nonumber &\quad + \left| \E f\left(u ((\mu-\tilde{\zeta}_\epsilon - \bar\zeta) t+\sigma B_t + \tilde{J}^1_t + \tilde{J}^2_t + J_t^3) \right) - \E f\left(u (\tilde{X}_t - t \bar\zeta) \right) \right| \\
	\begin{split}
	\label{eqn:expect-diff} &\leq u\|f'\|_{\infty} \left( t|\bar{\zeta} - (\zeta_\epsilon - \tilde{\zeta}_\epsilon)| + d_W(J^1_t, \tilde{J}^1_t)  + d_W(J^2_t, \tilde{J}^2_t) \right) \\
	 &\quad +  \left| \E f\left(u ((\mu-\tilde{\zeta}_\epsilon - \bar\zeta) t+\sigma B_t + \tilde{J}^1_t + \tilde{J}^2_t + J_t^3) \right) - \E f\left(u (\tilde{X}_t - t \bar\zeta) \right) \right|.
	\end{split}
\end{align}
We treat all terms in \eqref{eqn:expect-diff} individually.

\underline{Part (i)} 
The small jumps can be handled by noting \begin{align}
	\label{eqn:bias-smalljumps}d_W(J^1_t, \tilde{J}^1_t) & \leq \E|J^1_t| + \E|\tilde{J}^1_t| \leq \sqrt{\E|J^1_t|^2} + \sqrt{\E|\tilde{J}^1_t|^2}.
\end{align}
Since $J^1_t$ and $\tilde{J}^1_t$ have bounded jumps, we have $\E|J^1_t|^2,\E|\tilde{J}^1_t|^2 \to 0$ as $\epsilon \to 0$. 
Furthermore, $|\bar{\zeta} - (\zeta_\epsilon - \tilde{\zeta}_\epsilon)| \to 0$ as $\epsilon \to 0$. 

\underline{Part (ii)} 
As a next step, we study the medium sized jumps $J^2_t$. 
Consider the slightly more general process 
\begin{align*}
	J_t^{(a,b]} = \sum_{s\leq t} \Delta X_s \mathds{1}_{a < |\Delta X_s| \leq b},
\end{align*}
for $0<a<b<1$. Let $\tilde{J}_t^{(a,b]}$ be defined analogously based on $\tilde{X}_t$. 
These are compound Poisson processes, which can be written as 
\begin{align*}
	J_t^{(a,b]} = \sum_{i=1}^{N_t} U_i,\qquad \tilde{J}_t^{(a,b]} = \sum_{i=1}^{\tilde{N}_t} \tilde{U}_i,
\end{align*}
where $N_t$ is a Poisson counting process with intensity $\eta((a,b]) = \nu([-b,-a) \cup (a, b])$, and the $U_i$ are iid random variables with distribution $\frac{\nu(dz) \mathds{1}(a<|z|\leq b)}{\eta((a,b])}$. 
Vice versa, the same holds for $\tilde{N}_t$ and $\tilde{U}_i$ with $\tilde{\eta}((a,b]) = \tilde{\nu}([-b,-a)\cup (a,b])$. 
Then Theorem 10 and Proposition 3 of \cite{mariucci2017wasserstein} for $p=1$, yield
\begin{align}
	\nonumber d_W(J^{(a,b]}_t, \tilde{J}_t^{(a,b]}) &= d_W\left(\sum_{i=1}^{N_t} U_i, \sum_{i=1}^{\tilde{N}_t} \tilde{U}_i\right)\\
	 \label{eqn:wasserstein-1a}& \leq t{\eta}((a,b]) d_W(U_1, \tilde{U}_1) + t \left|\eta((a,b]) - \tilde{\eta}((a,b])\right|  \E|\tilde{U}_1|.
\end{align}
We compute 
\begin{align}
	\E|\tilde{U}_1| &= \frac{1}{\tilde{\eta}((a,b])} \left[ a \tilde{\nu}((a,b]) +\int_a^b \tilde{\nu}((z,b])dz + a \tilde{\nu}([-b,-a)) + \int_a^b \tilde{\nu}([-b, -z))dz \right] \nonumber \\
	&= a + \int_a^b \frac{\tilde{\eta}((z,b])}{\tilde{\eta}((a,b])}\, dz. \nonumber
\end{align}
Recall that $\tilde{\eta}((z,b]) = \sum_{m=1}^M (r_m^++r_m^-) (|z|^{-\alpha_m}-b^{-\alpha_m})$.
Then there exists a constant $\tilde{C}$ which is bounded on compacts in $\Theta$ and $L$, such that for $z<b/2$, and $\alpha=\alpha_1$, 
\begin{align}
	\frac{1}{\tilde{C}} |z|^{-\alpha}\leq\tilde{\eta}((z,b]) \leq \tilde{C} |z|^{-\alpha}. \label{eqn:expect-U}
\end{align}
In particular, this yields $\E|\tilde{U}_1| \leq \tilde{C} a^{1\wedge \alpha}$ for a potentially different constant $\tilde{C}$.
here and in the following, the constant $\tilde{C}$ may vary from line to line, and is bounded on compacts in $\theta$, $L$, and $\rho$. 

Furthermore, since $\nu$ and $\tilde{\nu}$ are sufficiently similar, 
\begin{align*}
	\eta((a,b]) &= \nu((a,\infty)) + \nu((-\infty, -a)) - \nu([-b,b]^c) \\
	&= \tilde{\eta}((a,b]) + \xi,
\end{align*}
for $|\xi| \leq 2L (a^{-\rho} +b^{-\rho}) \leq 4L a^{-\rho}$.
Thus, the second term in \eqref{eqn:wasserstein-1a} is of order $\mathcal{O}(t a^{(1\wedge\alpha)-\rho})$. 
Moreover, $|\eta((a,b])| \leq \tilde{C}(a^{-\alpha} + a^{-\rho}) =\mathcal{O}(a^{-\alpha})$ for small $a$, since $\rho<\alpha$.

We now consider the distance $d_W(U_1, \tilde{U}_1)$ occurring in \eqref{eqn:wasserstein-1a}, which can be expressed in terms of their cumulative distribution functions as 
\begin{align}
	\nonumber &\quad d_W(U_1, \tilde{U}_1) \\
	\nonumber &= \int_{-b}^{b} \left| P(U_1 \leq v) - P(\tilde{U}_1\leq v) \right| dv\\
	 \begin{split}
	\label{eqn:Wasserstein-U}	
	 &= \int_{-b}^{-a} \left| P(U_1 \leq v) - P(\tilde{U}_1\leq v) \right| du + \int_{a}^b \left| P(U_1> v) -P(\tilde{U}_1> v)  \right|dv  \\
	 &\qquad + 2a \left| P(U_1 \leq -a) - P(\tilde{U}_1\leq -a) \right|.
	\end{split}
\end{align}
For $-b\leq v< -a$, and $b\leq 1$, it holds 
\begin{align}
	\nonumber\left| P(U_1 \leq v) - P(\tilde{U}_1\leq v) \right| &= \left| \frac{\nu([-b,v])}{\eta((a,b])}-\frac{\tilde\nu([-b,v])}{\tilde\eta((a,b])} \right|\\
	&\leq \left| \frac{1}{\eta((a,b])} -  \frac{1}{\tilde{\eta}((a,b])} \right| \tilde{\nu}([-b,v]) \\
	\nonumber&\qquad  + \frac{1}{\eta((a,b])} \left| \nu([-b,v])-\tilde{\nu}([-b,v]) \right| \\
	&\leq \tilde{C} |v|^{-\alpha} \frac{\left| \eta((a,b])-\tilde{\eta}((a,b]) \right|}{ \left[ \eta((a,b]) \wedge \tilde{\eta}((a,b]) \right]^2}  +\frac{|\nu([-b,v])-\tilde{\nu}([-b,v])|}{ \eta((a,b]) \wedge \tilde{\eta}((a,b])}. \label{eqn:U-cdf-1}
\end{align}
Recall that $|\eta((a,b]) - \tilde{\eta}((a,b])| = \mathcal{O}(a^{-\rho})$. 
Furthermore, the assumed similarity of $\nu$ and $\tilde{\nu}$ implies that $|\nu([-b,v])-\tilde{\nu}([-b,v])| \leq L(|v|^{-\rho}+b^{-\rho})\leq 2L |v|^{-\rho}$, and \begin{align}
	\eta((a,b])\wedge \tilde{\eta}((a,b]) & \geq \tilde{\eta}((a,b]) - 2La^{-\rho} = \Omega(a^{-\alpha}) \label{eqn:eta-lower}
\end{align} 
as $a\to 0$, whenever $b\geq 2a$. In this case, for $-b\leq v\leq -a$, \begin{align}
	\label{eqn:U-cdf}\left| P(U_1 \leq v) - P(\tilde{U}_1\leq v) \right| & \leq \tilde{C}|v|^{-\alpha} a^{2\alpha-\rho} + \tilde{C} |v|^{-\rho} a^\alpha
	\leq \tilde{C} |v|^{-\rho} a^{\alpha}.
\end{align}
The analogous bound holds for $|P(U_1> v)-P(\tilde{U}_1> v)|$, when $a\leq v \leq b$. 
Now plug \eqref{eqn:U-cdf} into expression \eqref{eqn:Wasserstein-U} for the Wasserstein distance, to obtain for $a\to 0$ and $a\leq \frac{b}{2}$, 
\begin{align*}
	d_W(U_1, \tilde{U}_1) & \leq \tilde{C} \left(a^{\alpha}b^{1-\rho } +a^{\alpha+1-\rho}\right),
\end{align*}
where we used $\rho<1$.
Using \eqref{eqn:wasserstein-1a}, we may hence bound, \begin{align}
	\begin{split}
	\label{eqn:Jab-W1}d_W(J_t^{(a,b]}, \tilde{J}_t^{(a,b]}) & \leq \tilde{C} t \left(a^{-\alpha} (a^{\alpha}b^{1-\rho} + a^{\alpha+1-\rho}) + a^{(1\wedge\alpha)-\rho} \right) \\
	&\leq \tilde{C} t \left(b^{1-\rho}+a^{(1\wedge\alpha)-\rho}\right),
	\end{split}
\end{align}
This upper bound will be exploited in the rest of the proof. 
In particular, for $J_t^2 = J_t^{(\epsilon,1/u]}$ and $\epsilon$ small enough, 
\begin{align}
	\label{eqn:bias-medjumps}d_W(J_t^2, \tilde{J}_t^2) \leq \tilde{C} t u^{\rho-1}.
\end{align}

\underline{Part (iii)}
It remains to study the term in \eqref{eqn:expect-diff} due to the large jumps. 
Here, our approach is slightly different as we will not (only) bound a metric distance between $J_t^3$ and $\tilde{J}_t^3$. 
Define 
\begin{align*}
	f_{u,t}(x)=\E f(u(x + t(\mu-\tilde{\zeta}_\epsilon - \bar{\zeta}) + \sigma B_t + \tilde{J}_t^1 + \tilde{J}_t^2)),
\end{align*}
and we consider $\left|\E f_{u,t}(J_t^3)-\E f_{u,t}(\tilde{J}_t^3)\right|$, as suggested by \eqref{eqn:expect-diff}. 
Since $J_t^3$ is a L{\'e}vy process, Itô's formula yields\begin{align}
	\label{eqn:generator-J3}
	\E f_{u,t}(J_t^3) &= f_{u,t}(0)+\int_0^t \E \mathcal{J}^3 f_{u,t}(J_s^3)ds, \\
	\nonumber\mathcal{J}^{3}g(x)&= \int_{[-\frac{1}{u}, \frac{1}{u}]^c} \left[g(x+z)-g(x)\right]\,\nu(dz),
\end{align}
i.e., $\mathcal{J}^3$ is the infinitesimal generator of $J_t^3$. 
Analogously, we denote by $\tilde{\mathcal{J}}^3$ the generator of $\tilde{J}_t^3$. 
Then integration by parts yields, for any $x\in\R$, 
\begin{align*}
	&\quad \left|\int_{(1/u, \infty)} \left[f_{u,t}(x+z)-f_{u,t}(x)\right]\,(\nu-\tilde{\nu})(dz) \right| \\
	&= \Bigg| \left[f_{u,t}\left(x+\frac{1}{u}\right)-f_{u,t}(x)\right] \,\left[ \nu((1/u,\infty)) - \tilde\nu((1/u,\infty))  \right] \\
	&\quad + \int_{\frac{1}{u}}^\infty \left[ \nu((z,\infty)) - \tilde\nu((z,\infty))  \right] f'_{u,t}(x+z)\, dz\Bigg|  \\
	&\leq 2\|f\|_\infty L u^{\rho} + \int_{\frac{1}{u}}^1 L z^{-\rho} |f'_{u,t}(x+z)|\, dz + \left[\nu((1,\infty)) + \tilde{\nu}((1,\infty)) \right]\int_1^\infty |f'_{u,t}(x+z)|\, dz \\
	&\leq \tilde{C}\|f\|_\infty  u^{\alpha-\delta} + \tilde{C}u^{\rho} \int_{\frac{1}{u}}^1 |f'_{u,t}(x+z)|\, dz + \tilde{C} \int_1^\infty |f'_{u,t}(x+z)|\, dz .
\end{align*}
The same bound holds for the range of integration $z\in(-\infty, -1/u)$, such that 
\begin{align*}
		\left| \mathcal{J}^3 f_{u,t}(x) - \tilde{\mathcal{J}}^3f_{u,t}(x) \right| 
	\leq \tilde{C} u^\rho \left(\|f\|_\infty + \int_{-\infty}^\infty |f'_{u,t}(z)|\, dz \right).
\end{align*}
Now note that, 
\begin{align*}
	f'_{u,t}(x) &= u \, \E f'\left(u(x + t(\mu+\zeta_0) + \sigma B_t + \tilde{J}_t^1 + \tilde{J}_t^2)\right),
\end{align*}
such that by Fubini's theorem, 
\begin{align*}
	\int_{-\infty}^\infty |f'_{u,t}(z)|\, dz & \leq  \E \int_{-\infty}^\infty u\left|f'\left(u(z + t(\mu+\zeta_0) + \sigma B_t + \tilde{J}_t^1 + \tilde{J}_t^2)\right) \right|\, dz \\
	&= \E\int_{-\infty}^\infty |f'(v)|dv \qquad = \|f'\|_{L_1(\R)},
\end{align*}
where we performed a linear substitution in the second step. 
Hence, 
\begin{align}
	\left| \mathcal{J}^3 f_{u,t}(x) - \tilde{\mathcal{J}}^3f_{u,t}(x) \right| &\leq \tilde{C} u^{\rho} (\|f\|_\infty + \|f'\|_{L_1} ). \label{eqn:generator-J3-similarity}
\end{align}
Using this in \eqref{eqn:generator-J3}, 
\begin{align}
	\nonumber \E f_{u,t}(J_t^3) &= \int_0^t \E \tilde{\mathcal{J}}^3 f_{u,t} (J_s^3)\, ds + \mathcal{O}(u^{\rho}t) \\
	\nonumber &=\int_0^t \E\tilde{\mathcal{J}}^3 f_{u,t} (\tilde{J}_s^3) + \mathcal{O}\left(|\E\tilde{\mathcal{J}}^3f_{u,t}(J_s^3)- \E\tilde{\mathcal{J}}^3f_{u,t}(\tilde{J}_s^3)|\right)\, ds + \mathcal{O}(u^{\rho}t) \\
	\begin{split}
	\label{eqn:Ef-J3}
	&=\E f_{u,t}(\tilde{J}^3_t) + \mathcal{O}\left(u^{\rho} t(\|f\|_\infty + \|f'\|_\infty + \|f'\|_{L_1})\right) \\
	&\quad + \mathcal{O}\left( \|\tilde{\mathcal{J}}^3 f_{u,t}\|_\infty \int_0^t d_{TV}\left(J_s^{(1,\infty)},\tilde{J}_s^{(1,\infty)}\right)ds\right) \\
	&\quad + \mathcal{O}\left( \|(\tilde{\mathcal{J}}^3 f_{u,t})'\|_\infty \int_0^t d_{W}\left(J_s^{(\frac{1}{u},1]},\tilde{J}_s^{(\frac{1}{u},1]}\right)ds \right).
	\end{split}
\end{align}
We now study the latter two terms.

\underline{Part (iv)} 
The total variation distance can be bounded by noting that $J_t^{(1,\infty)}$ and $\tilde{J}_t^{(1,\infty)}$ admit only finitely many jumps. 
The number of their jumps is Poisson distributed, such that 
\begin{align*}
	d_{TV}(J_t^{(1,\infty)},0)= 1-P(J_t^{(1,\infty)}=0) 
	&= 1-\exp\left[-t\nu(( -1, 1 )^c)\right] \\
	&\leq t \, \nu( (-1, 1 )^c).
\end{align*}
In particular, 
\begin{align}
	d_{TV}(J_t^{(1,\infty)}, \tilde{J}_t^{(1,\infty)}) 
	&\leq t \left[ \nu( (-1, 1 )^c) + \tilde{\nu}( (-1, 1 )^c) \right] \leq t \tilde{C}. \label{eqn:J3-TV}
\end{align}
Moreover, 
\begin{align}
	\left| \tilde{ \mathcal{J}}^3 f_{u,t}(x) \right| 
	&\leq \int_{[-\frac{1}{u},\frac{1}{u}]^c} |f_{u,t}(x+z)-f_{u,t}(x)| \tilde{\nu}(dz) \nonumber \\
	&\leq \|f_{u,t}\|_\infty \tilde{\nu}([-1/u, 1/u]^c)\nonumber \\
	&\leq \tilde{C} u^\alpha \|f\|_\infty. \label{eqn:J3-upper}
\end{align}
Via the same argument, we also obtain 
\begin{align}
	\left| \frac{d}{dx}\tilde{ \mathcal{J}}^3 f_{u,t}(x) \right| 
	&= \left|\int_{[-\frac{1}{u},\frac{1}{u}]^c} f'_{u,t}(x+z)-f'_{u,t}(x) \tilde{\nu}(dz) \right| \leq \tilde{C} u^{\alpha+1} \|f'\|_\infty. \label{eqn:J3-deriv-upper}
\end{align}
From \eqref{eqn:Jab-W1}, we know that 
\begin{align*}
	\int_0^t d_W\left(J_s^{(\frac{1}{u},1]}, \tilde{J}_s^{(\frac{1}{u},1]}\right)\, ds 
	\leq \tilde{C} \int_0^t s\, ds \leq \tilde{C} t^2.
\end{align*}
In combination with \eqref{eqn:Ef-J3}, we thus obtain 
\begin{align}
	\nonumber \left|\E f_{u,t}(J_t^3) - \E f_{u,t}(\tilde{J}_t^3)\right| 
	& \leq t u^{\rho}  \tilde{C} (\|f\|_\infty +\|f'\|_\infty +\|f'\|_{L_1}) \\
	\nonumber &\quad + \tilde{C} t^2u^{\alpha}\|f\|_\infty + \tilde{C} t^2u^{\alpha+1}\|f'\|_\infty \\
	\label{eqn:bias-largejumps}&\leq \tilde{C} (\|f\|_\infty +\|f'\|_\infty +\|f'\|_{L_1})(tu^{\rho}+t^2u^{\alpha+1}).
\end{align}

\underline{Part (v)}
Now putting \eqref{eqn:bias-smalljumps}, \eqref{eqn:bias-medjumps}, and \eqref{eqn:bias-largejumps} into \eqref{eqn:expect-diff}, and letting $\epsilon\to 0$, 
\begin{align}
	\label{eqn:bias-asymmetric} \left| \E f(uX_t) - \E f(u\tilde{X}_t + ut\zeta_0) \right| 
	&\leq \tilde{C} (\|f\|_\infty +\|f'\|_\infty +\|f'\|_{L_1})(tu^{\rho}+t^2u^{\alpha+1}).
\end{align}
It can be checked that the upper bounds which are summarized in the constant $\tilde{C}$ all satisfy the desired uniformity on compacts in $\boldsymbol{\alpha}$, $\boldsymbol{r}$, $L$, and $\rho-\alpha<0$. 
This concerns the lines \eqref{eqn:expect-U}, \eqref{eqn:U-cdf-1}, \eqref{eqn:eta-lower}, \eqref{eqn:generator-J3-similarity}, \eqref{eqn:J3-TV}, \eqref{eqn:J3-upper}, \eqref{eqn:J3-deriv-upper}.
\end{proof}

\begin{proof}[Proof of Corollary \ref{cor:bias-complete}]
	Assume $f(0)=0$ without loss of generality.
	A Taylor expansion yields, for any $a\in\R$,
	\begin{align*}
		|\E f(u(\tilde{X}_t+ta)) - \E f(u\tilde{X}_t)| &\leq |u t a \E f'(u\tilde{X}_t)| + \|f''\|_\infty t^2u^2a^2.
	\end{align*}
	We denote $\tilde{X}_t = \sigma B_t + \tilde{J}_t$, where $\tilde{J}_t$ is the purely discontinuous component of $\tilde{X}$. Introduce for any function $g$ the notation $g_{[u]}(x) = \E g(u\sigma B_t + x)$. 
	Then for any $k$-th derivative, $\|g_{[u]}^{(k)}\|_\infty \leq \|g^{(k)}\|_\infty$. 
	In particular, by Lemma \ref{lem:moments}, \begin{align*}
		|\E f'(u\tilde{X}_t) | = |\E f'_{[u]}(u\tilde{J}_t)| \leq t u^\alpha (1+\log(u)) \left( \|f'\|_\infty +\|f''\|_\infty+\|f'''\|_\infty\right) \tilde{C},
	\end{align*}
	such that 
	\begin{align}
		\label{eqn:bias-drift} 
		\begin{split}
		&\quad |\E f(u(\tilde{X}_t+ta)) - \E f(u\tilde{X}_t)| \\ &\leq t^2 u^{2\vee (\alpha+1)} (1+\log(u))\left( \|f'\|_\infty +\|f''\|_\infty+\|f'''\|_\infty\right) (|a|+|a|^2) \tilde{C}.
		\end{split}
	\end{align}
	
	Moreover, $|\E f(uX_t)-\E f(u(\tilde{X}_t+t\mu-t\bar{\zeta}))| \leq \tilde{C}(tu^{\rho} + t^2u^{\alpha+1})$ from Lemma \ref{lem:bias-bound}. 
	Applying \eqref{eqn:bias-drift} for the drift $a=\mu+\zeta_0$, this yields \eqref{eqn:bias-withdrift}.
\end{proof}

\begin{proof}[Proof of Lemma \ref{lem:moment-clt}]
	All summands $\f(\Delta_{n,i} X)$ are iid and bounded and $\tilde{\Lambda}_n^{-1} / \sqrt{n}\to 0$, such that the Lindeberg-Feller condition for triangular arrays of independent r.v.s is satisfied \cite[Thm.\ 2.4.5]{durrett2010probability}.
	Moreover, the bias is of order 
		$|\E \f({u}_n\Delta X_{t_i}) - \E\f_j({u}_n\tilde{Z}_h)| = \mathcal{O}(h_nu_n^{\rho})$
	by Corollary \ref{cor:bias-complete}. 
	If $\rho<\alpha/2$, this is small enough to ensure 
		$\Lambda_n^{-1} \sqrt{n} |\E \f(\Delta_{n,i} X) - \E_\theta \f(u_n \tilde{Z}_h)| = o(1)$.
	Hence, the bias is asymptotically negligible.
	
	It thus suffices to check the asymptotic covariance structure. 
	Denote $f_{j,k}(x) = f_j(x)f_k(x)$. 
	Then $f_{j,k}$ is smooth and vanishes on $[-\eta,\eta]$ unless $j=1=k$. 
	Moreover, $f_{1,1}(0)=f_{1,1}'(0)=f_{1,1}''(0)=0$ and $f_{1,1}^{(4)}(0)=6f_1''(0)^2$. 
	Corollary \ref{cor:bias-complete} and Lemma \ref{lem:moments} yield 
	\begin{align*}
			\E f_{j,k}(u_n\Delta_{n,i} X) 
		&= \E_\theta f_{j,k}(u_n\tilde{Z}_h) + \mathcal{O}(h_nu_n^{\rho}) \\
		&= u_n^\alpha h \left(r_1^+\mathcal{J}_{\alpha}^+ f_{j,k}(0)+ r_1^-\mathcal{J}_{\alpha}^- f_{j,k}(0)\right) + o(u_n^\alpha h), \quad (j,k)\neq (1,1),\\
			\E f_{1,1}(u_n \Delta_{n,i} X) 
		&= \frac{3}{4} u_n^4 h^2 \sigma^4 f_1''(0)^2 + o(u_n^4h^2) + \mathcal{O}(u_n^\alpha h) + \mathcal{O}(u_n^\rho h) \\
		&= \frac{3}{4} \sigma^4 u_n^4 h^2 f_1''(0)^2 + o(u_n^4 h^2).
	\end{align*}
	To compute the asymptotic covariance, we further determine 
	\begin{align*}
		\left(\E f_1(u_n \Delta_{n,i} X) \right)^2 
		&= \left( hu^2 \frac{\sigma^2}{2} f''(0) + o(hu^2) + \mathcal{O}(hu^\rho) \right)^2 \\
		&= \frac{h^2u^4 \sigma^4}{4} f''(0)^2 + o(h^2u^4),
	\end{align*}
	and for $j\geq 2,k\geq 1$, \begin{align*}
		\E f_j(u_n \Delta_{n,i} X)\, \E f_k(u_n \Delta_{n,i} X) 
		&= \mathcal{O}(u_n^\alpha h) \cdot  \mathcal{O}(u_n^2 h) = o(u_n^\alpha h).
	\end{align*}
	These approximations can be summarized as 
	\begin{align*}
		\Cov(\f(\Delta_{n,i} X))_{j,k} 
		&= \begin{cases}
			\frac{\sigma^4}{2}u_n^4 h^2 f''(0)^2 + o(u_n^4 h^2),& j=k=1\\
			u_n^\alpha h \left(r_1^+\mathcal{J}_{\alpha}^+ f_{j,k}(0)+ r_1^-\mathcal{J}_{\alpha}^- f_{j,k}(0)\right) + o(u_n^\alpha h), & \text{otherwise},
		\end{cases}
	\end{align*}
	This scaling behavior yields $\Cov_\theta( \tilde{\Lambda}_n^{-1}(\theta)\f(\Delta_{n,i} X) ) \to \Sigma(\theta)$ as $n\to 0$, and thus the desired central limit theorem.
\end{proof}

\begin{proof}[Proof of Lemma \ref{lem:expect-derivative}]
	First, assume $f$ to be a Schwartz function with Fourier transform $\hat{f}(\lambda)$. Then \begin{align*}
		\E  f(u\tilde{X}_h) = \frac{1}{2\pi}\int \hat{f}(\lambda/u) e^{-h\psi_\theta(\lambda)}\, d\lambda, 
	\end{align*}
	where $\psi_\theta$ is the Lévy symbol of $\tilde{X}_h$, i.e.\ $\E_\theta\exp(i\lambda \tilde{X}_h) = \exp(-h\psi_\theta(\lambda))$. In particular, for any entry $\theta_j$ of the parameter vector $\theta$, \begin{align*}
			\partial_{\theta_j} \E_\theta f(u\tilde{X}_h) 
		&= -h \int \hat{f}(\lambda) \left( \partial_{\theta_j} \psi_\theta(u\lambda) \right) e^{-h\psi_\theta(u\lambda)}  \, d\lambda.
	\end{align*}
	Integration and differentiation may be exchanged because $f$ is a Schwartz function and $\psi$ has polynomial growth.
	In particular, via the Lévy-Khintchine formula, the Lévy symbol may be determined as
	\begin{align*}
		\psi_\theta(u\lambda) 
		&= \frac{u^2\sigma^2\lambda^2}{2} + \int \left[e^{i u\lambda z} - 1 - iu\lambda \trunc(z)\right]\, \tilde{\nu}(dz)\\
		&= \frac{u^2\sigma^2\lambda^2}{2} - i \lambda \int_{|z|\geq \frac{1}{u}} \left[u\trunc(z) - \trunc(uz) \right]\tilde{\nu}(dz)  \\
		&\quad + \sum_{m=1}^M \alpha_m u^{\alpha_m} \lambda^{\alpha_m} \int \frac{e^{i z} - 1 - i\trunc(z)}{|z|^{1+\alpha_m}} \left( r_m^+ \mathds{1}_{z>0} + r_m^-\mathds{1}_{z<0} \right)\, dz.
	\end{align*}
	The second term appears because the Lévy measure $\tilde{\nu}$ is allowed to be asymmetric. 
	In its expression, we used that $\trunc(z)=z$ for $z\in(-1,1)$, and denote 
	\begin{align}
		\overline{\trunc}_u 
		&= \int_{|z|\geq \frac{1}{u}} \left[u\trunc(z) - \trunc(uz) \right]\tilde{\nu}(dz) \nonumber\\
		&= u \int_{|z|\geq 1} \trunc(z)\, \tilde{\nu}(dz) + u \int_{\frac{1}{u}\leq |z| < 1} z \, \tilde{\nu}(dz) - \int_{|z|\geq \frac{1}{u}} \trunc(uz) \,\tilde{\nu}(dz). \label{eqn:expect-deriv-1}
	\end{align}
	Hence, by inverting the Fourier transform,
	\begin{align}
		\partial_{\theta_j} \E_\theta f(u\tilde{X}_h) 
		= h\,\E_\theta\left[ \partial_{\theta_j} \left( \frac{\sigma^2 u^2}{2} f'' + \sum_{m=1}^M u^{\alpha_m} (r_m^+ \mathcal{J}_{\alpha_m}^+ f + r_m^- \mathcal{J}_{\alpha_m}^-) - \overline{\trunc}_u f'\right)(u\tilde{X}_h) \right]. \label{eqn:expect-deriv-2}
	\end{align}
	
	So far, we assumed $f$ to be a Schwartz function, but the right hand side of \eqref{eqn:expect-deriv-2} makes sense whenever $f\in\mathcal{C}^2$.
	We can extend the whole equation \eqref{eqn:expect-deriv-2} to this case by approximating $f$ suitably with a sequence of Schwartz functions $f_n$, such that $\sup_{|x|\leq K} |f^{(k)}_n(x)-f^{(k)}(x)|\to 0$ as $n\to\infty$ for each $K>0$, and $k=0,1,2$, and $\sup_n \|f_n^{(k)}\|_\infty < \infty$. 
	Hence, standard arguments allow us to pass to the limit on both sides of the equation \eqref{eqn:expect-deriv-2}
	
	To handle the asymmetry term $\bar{\trunc}_u$, we exploit \eqref{eqn:expect-deriv-1} to derive
	\begin{align*}
		\left|\partial_{r_m^\pm} \bar{\trunc}_u\right| 
		& \leq  u \|\trunc\|_\infty \int_1^\infty \alpha_m |z|^{-1-\alpha_m}dz + u \int_{\frac{1}{u}}^1 \alpha_m |z|^{-\alpha_m} \, dz + \|\trunc\|_\infty \int_{\frac{1}{u}}^{\infty} \alpha_m|z|^{-1-\alpha_m} dz \\
		&\leq u\|\trunc\|_\infty + u \left|\int_{\frac{1}{u}}^1 \alpha_m |z|^{-\alpha_m} \, dz\right| + \|\trunc\|_\infty u^{\alpha_m}. 
	\end{align*}
	The second integral can be bounded as follows.
	For any $\epsilon\in(0,1)$ and any $p\neq 1$, there is a $\tilde{p}$ between $p$ and $1$ such that
	\begin{align*}
		\left|\int_{\epsilon}^{1} |z|^{-p}\, dz \right|
		& = \frac{1}{|1-p|} |\epsilon^{1-p}-\epsilon^0| \\
		& =  \frac{|1-p|}{|1-p|} |\epsilon^{1-\tilde{p}} \log(\epsilon)|  
		\quad \leq |\log\epsilon|\epsilon^{(1-p)\wedge 0}.
	\end{align*}
	By continuity, the same bound holds for $p=1$.
	Thus, we obtain 
	\begin{align*}
		\left|\partial_{r_m^\pm} \bar{\trunc}_u\right| 
		&\leq u\|\trunc\|_\infty + \alpha_m |\log u| u^{\alpha_m\vee 1} + \|\trunc\|_\infty u^{\alpha_m} \\
		&\leq \tilde{C}u^{\alpha_m\vee 1} (1+|\log u|).
	\end{align*}
	Similarly,
	\begin{align*}
		\left|\partial_{ \alpha_m} \bar{\trunc}_u\right| 
		& \leq u\|\trunc\|_\infty (r_m^++r_m^-)\int_1^\infty \frac{\alpha_m |\log z| + 1}{|z|^{1+\alpha_m}}\, dz + u (r_m^++r_m^-) \int_{\frac{1}{u}}^{1} \frac{\alpha_m|\log z| + 1}{|z|^{\alpha_m}}\, dz  \\
		&\quad + \|\trunc\|_\infty (r_m^++r_m^-) \int_{\frac{1}{u}}^\infty \frac{\alpha_m |\log z| + 1}{|z|^{1+\alpha_m}}\, dz\\
		&\leq \tilde{C} (1+|\log u|)^2 u^{\alpha_m\vee 1}.
	\end{align*}
	Note also that $\partial_{\sigma^2} \overline{\trunc}_u=0$.
	
	For specific partial derivatives, we thus have shown that
	\begin{align}
		\begin{split}
			\partial_{\sigma^2} \E_\theta f(u\tilde{X}_h) 
		&= h \frac{u^2}{2} \E_\theta f''(u\tilde{X}_h), \\
			\partial_{r_m^\pm} \E_\theta f(u\tilde{X}_h) 
		&= h u^{\alpha_m} \E_\theta \mathcal{J}_{\alpha_m}^\pm f(u\tilde{X}_h) + \mathcal{O}\left(hu^{\alpha_m\vee 1} \log u\right) \E_\theta f'(u\tilde{X}_h),\\
			\partial_{\alpha_m} \E_\theta f(u\tilde{X}_h) 
		&= h u^{\alpha_m} \E_\theta \left( \frac{d}{d\alpha_m}(r_m^+\mathcal{J}_{\alpha_m}^+f +r_m^-\mathcal{J}_{\alpha_m}^-f)(u\tilde{X}_h) \right) \\
		&\quad + hu^{\alpha_m} \log u \E_\theta \left( (r_m^+\mathcal{J}_{\alpha_m}^+f +r_m^-\mathcal{J}_{\alpha_m}^-f)(u\tilde{X}_h) \right) \\
		&\quad + \mathcal{O}\left(hu^{\alpha_m\vee 1} (\log u)^2\right) \E_\theta f'(u\tilde{X}_h).
		\end{split} \label{eqn:expect-deriv-3}
	\end{align}
	For fixed $f$, the functions $f''$, $\mathcal{J}_{\alpha_m}^\pm f$ and $\partial_{\alpha_m} \mathcal{J}_{\alpha_m}^\pm f$ are bounded, uniformly on compacts in $\theta$. 
	Moreover, $P_\theta(|u\tilde{X}_h|>\eta)\to 0$ uniformly on compacts in $\Theta$ for any $\eta$, as established in the proof of Lemma \ref{lem:moments}. 
	Therefore, $\E_\theta f''(u\tilde{X}_h) \to f''(0)$ uniformly on compacts as $h\to 0$, as well as 
		$\E_\theta \mathcal{J}_{\alpha_m}^\pm f(u\tilde{X}_h) \to \mathcal{J}_{\alpha_m}^\pm f(0)$ 
	and 
		$\E_\theta \partial_{\alpha_m}\mathcal{J}_{\alpha_m}^\pm f(u\tilde{X}_h) \to \partial_{\alpha_m}\mathcal{J}_{\alpha_m}^\pm f(0)$.
	This completes the proof of \eqref{eqn:expect-deriv-result-1}, and \eqref{eqn:expect-deriv-result-2} follows analogously by applying a linear transformation to \eqref{eqn:expect-deriv-3}.
	Finally, \eqref{eqn:expect-deriv-result-3} is a consequence of \eqref{eqn:expect-deriv-3} upon noting that $\E_\theta f''(u\tilde{X}_h)=\mathcal{O}(hu^\alpha)$, see Lemma \ref{lem:moments}.
\end{proof}

\begin{proof}[Proof of Corollary \ref{cor:moment-derivative}]
	Since $f_1'$ is bounded, \eqref{eqn:expect-deriv-result-1} shows that \begin{align*}
		|\partial_{r_m^\pm} \E_\theta f_1(u\tilde{X}_h)| = o(hu^2),\quad |\partial_{\alpha_m} \E_\theta f_1(u\tilde{X}_h)| = o(hu^2).
	\end{align*}
	This corresponds to the entries $A(\theta)_{1,k}=0$ for $k\geq 2$.
	For $j\geq 2$, we have $\E_\theta f'_j(u\tilde{X}_h) = \mathcal{O}(hu^\alpha)$ by virtue of Lemma \ref{lem:moments}, since $f_j$ vanishes near zero. Hence, since $\alpha_m>\alpha/2$ and $u\leq \mathcal{O}(\sqrt{h})$, \begin{align*}
		\mathcal{O}\left(hu^{\alpha_m\vee 1} (\log u)^2\right) \E_\theta f'_j(u\tilde{X}_h) = \mathcal{O}(h^2u^{\alpha + (\alpha_m\vee 1)} (\log(u)^2)) \leq o(hu^{\alpha_m}).
	\end{align*}
	This corresponds to the entries $A(\theta)_{j,1}=0$ for $j\geq 2$.
	In combination with Lemma \ref{lem:expect-derivative}, this suffices to establish the convergence \eqref{eqn:deriv-conv}.
\end{proof}

\begin{proof}[Proof of Lemma \ref{lem:estimating-consistency}]
	Denote the estimating equation \eqref{eqn:def-GMM} as $F_n(\hat{\theta}_n)=0$, for 
	\begin{align}
		F_n(\theta) = \frac{1}{n} \sum_{i=1}^n \f(u_n \Delta_{n,i}X) - \E_{\hat{\theta}_n} \f(u_n \tilde{Z}_{h_n}). \label{eqn:F-sum}
	\end{align}
	Let $\theta_0$ be the true parameters, and reparameterize $\theta = \theta_0 + \Gamma_n(\theta_0) \bar{\Lambda}^{-1}_n(\theta_0)T$ for $T=\bar{\Lambda}_n(\theta_0)\Gamma_n^{-1}(\theta_0)(\theta-\theta_0)$, and let
	\begin{align*}
		\bar{F}_n(T)=\tilde{\Lambda}_n^{-1}(\theta_0)F_n\left(\theta_0 + \Gamma_n(\theta_0)\bar{\Lambda}^{-1}_n(\theta_0)T\right).
	\end{align*} 
	This is well defined whenever $T \in B_{d_n}(0)$, for $d_n= {c} \sqrt{h_n}u_n^{\alpha_M-\frac{\alpha_1}{2}}/(\log u_n)^3\to 0$, and $c>0$ sufficiently small.
	In this reparameterized model, we need to show that there exists a sequence of random vectors $\hat{T}_n\in B_{d_n}$ such that $\bar{F}_n(\hat{T}_n)=0$ for large $n$, and $\Gamma_n(\theta_0) \bar{\Lambda}_n^{-1}(\theta_0) \hat{T}_n\to 0$. 
	This will imply that $\|\hat{\theta}_n- \theta_0\| \leq C/(\log u_n)^2$ for a sufficiently large factor $C$.
	
	We know from Lemma \ref{lem:moment-clt} that \begin{align*}
		\bar{F}_n(0)=\tilde{\Lambda}_n^{-1}(\theta_0) F_n(\theta_0)
		=\mathcal{O}_P\left(\frac{1}{\sqrt{n}}\right) = o(d_n).
	\end{align*}
	Furthermore, 
	\begin{align*}
		\D_T \bar{F}_n(T)=\begin{pmatrix}
			\partial_{T_1} & \ldots & \partial_{T_{3M+1}} 
		\end{pmatrix} \bar{F}_n(T) = \tilde{\Lambda}_n^{-1}(\theta_0) \D_\theta F_n(\theta_0 + {\Gamma}_n \bar{\Lambda}_n^{-1} T) \Gamma_n(\theta_0)\bar{\Lambda}_n^{-1}(\theta_0).
	\end{align*}
	By Corollary \ref{cor:moment-derivative}, $\tilde{\Lambda}_n^{-1}(\theta) \D_\theta F_n(\theta) \Gamma_n(\theta) \bar{\Lambda}_n^{-1}(\theta) \to A(\theta)$ locally uniformly, and it can be checked that $\theta\mapsto A(\theta)$ is continuous. 
	Moreover, the definitions of $\tilde{\Lambda}_n,\bar{\Lambda}_n$, and $\Gamma_n$ readily yield, as $n\to\infty$,
	\begin{align}
		\begin{split}
		& \sup_{T\in \overline{B}_{d_n}(0)} \|\tilde\Lambda_n^{-1}(\theta_0)\tilde\Lambda_n(\theta_0 + \Gamma_n(\theta_0)\bar{\Lambda}_n^{-1}(\theta_0)T) - \mathbf{I}_{3M+1}\| \\
		&\leq \sup_{\|\theta-\theta_0\|\leq \frac{C}{(\log u_n)^2}} \|\tilde\Lambda_n^{-1}(\theta_0)\tilde\Lambda_n(\theta) - \mathbf{I}_{3M+1}\|\to 0, \\
		&\sup_{T\in \overline{B}_{d_n}(0)} \|\bar\Lambda_n^{-1}(\theta_0)\bar\Lambda_n(\theta_0 + \Gamma_n(\theta_0)\bar{\Lambda}_n^{-1}(\theta_0)T) - \mathbf{I}_{3M+1}\| \to 0, \\
		&\sup_{T\in \overline{B}_{d_n}(0)} \| \Gamma_n^{-1}(\theta_0 + \Gamma_n(\theta_0)\bar{\Lambda}_n^{-1}(\theta_0)T)\Gamma_n(\theta_0) - \mathbf{I}_{3M+1}\| \to 0.
		\end{split} \label{eqn:lambda-uniform-1}
	\end{align}
	Here, we denote by $\|\cdot\|$ the spectral norm of a matrix, i.e.\ $\|A\|^2$ is the largest absolute eigenvalue of the symmetrized matrix $A^TA$, and $\mathbf{I}_d$ denotes the $d\times d$ identity matrix. 
	Thus, 
	\begin{align*}
		\sup_{T\in \overline{B}_{d_n}(0)} \| \D_T \bar{F}_n(T) - A(\theta_0)\| &\to 0.
	\end{align*}
	
	Now we apply \cite[Lemma 6.2]{jacod2017review} to establish the existence of a solution $\hat{T}_n\in B_{d_n^*}(0)$ of the equation $\bar{F}_n(\hat{T}_n)=0$. 
	Let $\lambda = \frac{1}{2} \|A(\theta_0)^{-1}\|^{-1}$, and denote by $C_n$ the event 
	\begin{align*}
		C_n &= \left\{ \sup_{T\in \overline{B}_{d_n}(0)} \| \D_T \bar{F}_n(T) - A(\theta_0)\| \leq \lambda \right\} \quad \cap \quad \left\{ \left\|\bar{F}_n(0) \right\| \leq \lambda d_n \right\}.
	\end{align*}
	Since the first set is deterministic, and since $\|\bar{F}_n(0)\|/d_n\pconv 0$, we have $P(C_n)\to 1$. 
	On the set $C_n$, it holds that $0\in \overline{B}_{\lambda d_n}(\bar{F}_n(0))$. 
	Then Lemma 6.2 of \cite{jacod2017review} with $y=0, f=\bar{F}_n$ and $r=d_n$, states that there exists a unique point $\hat{T}_n\in \overline{B}_{d_n}(0)$ which solves $\bar{F}_n(\hat{T}_n)=0$. 
	
	Returning to the original parametrization, we conclude there exists a random variable $\hat{\theta}_n$ such that with probability at least $1-P(C_n) \to 0$, $\hat{\theta}_n$ solves the estimating equation and $\hat{\theta}_n - \theta_0 \in \Gamma_n(\theta_0)\bar{\Lambda}_n^{-1}(\theta_0) \overline{B}_{d_n}(0)$, i.e.\ $\hat{\theta}_n-\theta_0 = \mathcal{O}_P(1/\log u_n)$. 
	Theorem \ref{thm:estimating-clt} below establishes that any consistent sequence $\hat{\theta}_n^*$ converges at a rate faster than $1/\log u_n$, 
	such that 
		$\hat{T}_n^* = \Gamma_n(\theta_0)^{-1}\bar{\Lambda}_n^{-1}(\theta_0)(\hat{\theta}_n^* - \theta_0)\in \overline{B}_{d_n}(0)$
	eventually. 
	Hence, the uniqueness of $\hat{T}_n$ on $\overline{B}_{d_n^*}(0)$ implies the uniqueness of $\hat{\theta}_n$, i.e.\ $P(\hat{\theta}_n^*\neq \hat{\theta}_n) =P(\hat{T}_n^*\neq \hat{T}_n) \to 0$.
\end{proof}

\begin{proof}[Proof of Theorem \ref{thm:estimating-clt}]
Denote the estimating equation as $F_n(\theta)=0$, for $F_n(\theta)$ as in \eqref{eqn:F-sum}. 
The mean value theorem yields \begin{align*}
	0 = \tilde{\Lambda}_n^{-1}(\theta_0) F_n(\hat{\theta}_n) 
	&= \tilde{\Lambda}_n^{-1}(\theta_0)F_n(\theta_0) +\left[\tilde{\Lambda}_n^{-1}\widetilde{F_n} \Gamma_n \bar{\Lambda}_n^{-1}\right] \bar{\Lambda}_n \Gamma_n^{-1}  (\hat{\theta}_n-\theta_0),
\end{align*}
where $(\widetilde{F}_n)_{j,k} = \partial_{\theta_k} (F_n)_j(\tilde{\theta}^j)$ for some $\tilde{\theta}^j$ on the line segment between $\theta_0$ and $\hat{\theta}_n$.
Denote by $R_n\subset \Omega$ the event that $A_n=\tilde{\Lambda}_n(\theta_0)^{-1}\widetilde{F_n} \Gamma_n(\theta_0)\bar{\Lambda}_n(\theta_0)^{-1}$ is regular, and introduce furthermore the matrices \begin{align*}
	A_n^j = \tilde{\Lambda}(\theta_0)^{-1} \D_\theta F_n(\tilde{\theta}^j) \Gamma_n(\theta_0) \bar{\Lambda}_n(\theta_0)^{-1},\quad j=1,\ldots, 3M+1.
\end{align*}
That is, the $j$-th row of $A_n$ and $A_n^j$ coincide, $(A_n)_{j,k}=(A_n^j)_{j,k}$.
Now note that $\|\tilde{\theta}^j-\theta_0\|\leq \|\hat{\theta}-\theta_0\|=\mathcal{O}_P(1/(\log u_n)^2)$,
and for any $C>0$, as in \eqref{eqn:lambda-uniform-1},
\begin{align}
	\begin{split}
	& \sup_{\|\theta-\theta_0\|\leq \frac{C}{(\log u_n)^2}} \|\tilde\Lambda_n^{-1}(\theta_0)\tilde\Lambda_n(\theta) - \mathbf{I}_{3M+1}\| \\
	& \sup_{\|\theta-\theta_0\|\leq \frac{C}{(\log u_n)^2}} \|\bar\Lambda_n^{-1}(\theta_0)\bar\Lambda_n(\theta) - \mathbf{I}_{3M+1}\| \\
	& \sup_{\|\theta-\theta_0\|\leq \frac{C}{(\log u_n)^2}} \|\Gamma_n^{-1}(\theta_0)\Gamma_n(\theta) - \mathbf{I}_{3M+1}\|.
	\end{split}\label{eqn:lambda-uniform-2}
\end{align}
Together with the locally uniform convergence of Corollary \ref{cor:moment-derivative}, this yields $A_n^j\pconv A(\theta_0)$ for each $j$, and thus $A_n\pconv A(\theta_0)$.

In particular, $P(R_n)\to 1$, and on the set $R_n$, we may rewrite 
\begin{align*}
		\sqrt{n}\bar{\Lambda}_n \Gamma_n^{-1}(\theta_0) (\hat{\theta}_n-\theta_0) 
	&= -\sqrt{n}  A_n^{-1} \tilde{\Lambda}_n^{-1} F_n(\theta_0).
\end{align*}
But $\sqrt{n} \tilde{\Lambda}_n^{-1} F_n(\theta_0) \wconv \mathcal{N}(0, \Sigma(\theta_0))$ by Lemma \ref{lem:moment-clt}, and $A_n^{-1}\to A^{-1}(\theta)$ in probability, such that Slutsky's lemma completes the proof.
\end{proof}

\begin{proof}[Proof of Proposition \ref{prop:fisher}]
We show how to adjust the proof of \cite{ait2012identifying} to consider the off-diagonal entries. 
Denote by $\varphi_\alpha$ the density of a symmetric $\alpha$-stable random variable, standardized to have L{\'e}vy measure $\alpha |x|^{-1-\alpha}dx$. 
This is the same parametrization as implied by \eqref{eqn:def-stablesum}. 
Furthermore, let $\varphi$ be the density of a standard normal distribution. 
Then the probability density of $\tilde{Z}_h$ is given by the convolution 
\begin{align*}
	p_h(x) = \int \frac{1}{\sqrt{\sigma^2h}} \varphi\left( \frac{x - (r h )^{\frac{1}{\alpha}}y}{\sqrt{\sigma^2h}} \right) \varphi_\alpha(y)\, dy.
\end{align*}
Now introduce the terms 
\begin{align*}
	w_h = (rh)^{\frac{1}{\alpha}} / \sqrt{\sigma^2h},\quad&\quad v_h=\frac{1}{\alpha(2-\alpha)} \left( 2+\frac{\log(r/\sigma^2)}{\log(1/w_h)} \right),
\end{align*}
and
\begin{align*}
	S_h(x) &= \int \varphi(x-w_hy) \varphi_\alpha(y) \,dy &&=\sqrt{\sigma^2h}\cdot p_h(x \sqrt{\sigma^2 h}) , \\
	R_h^0(x) &= \frac{1}{w_h^\alpha} \int \varphi(x-w_hy) ( \varphi_\alpha(y) + y \partial_y \varphi_\alpha(y))\, dy &&=\frac{-r \alpha \sqrt{\sigma^2h}}{w_h^\alpha} \cdot \frac{d}{dr}p_h(x \sqrt{\sigma^2h}),\\
	R_h^1(x) &= \frac{1}{w_h^\alpha \log(1/w_h)} \int \varphi(x-w_hy)\partial_\alpha \varphi_\alpha(y)\, dy &&\\
	\leadsto  & \qquad{w_h^\alpha \log(1/w_h)} R_h^1(x) - {w_h^\alpha v_h \log(1/w_h)} R_h^0(x) &&=\sqrt{\sigma^2 h}\frac{d}{d\alpha} {p_h(x \sqrt{\sigma^2 h})},\\
	J_h^{l,m} &= \int\frac{R_h^l(x) R_h^m(x)}{S_h(x)} dx,\qquad l,m\in\{0,1\}.
\end{align*}
Some technical integral transformations, explained in more detail by \cite{ait2012identifying} (cf. (A.3) therein), establish that 
\begin{align*}
	\mathcal{I}_h^{r,r} &= \frac{w_h^{2\alpha}}{r^2\alpha^2} J_h^{0,0}, \\
	\mathcal{I}_h^{\alpha,\alpha} &= \int \frac{w_h^{2\alpha} \log(1/w_h)^2( R_h^1(x) - v_h R_h^0(x) )^2}{S_h(x)}dx\\
	&= w_h^{2\alpha} \log(1/w_h)^2(J_h^{1,1}(x) - 2  v_h J_h^{1,0}(x) + v_h^2  J_h^{0,0}(x)),\\
	\mathcal{I}_h^{\alpha,r} &= \int \frac{w_h^{2\alpha} \frac{-R_h^0(x)}{r\alpha} \log(1/w_h) \left( R_h^1(x) - v_h R_h^0(x) \right) }{S_h(x)} \\
	&= \frac{w_h^{2\alpha} \log(1/w_h)}{r\alpha} \left( v_h J_h^{0,0}(x) - J_h^{1,0}(x) \right).
\end{align*}
The main workload of the proof given by \cite{ait2012identifying} derives the limiting behavior of $J_h^{l,m}$ as $h\to 0$. 
They show that 
\begin{align*}
	J_h^{0,0}/\psi_h \to \alpha^4, \quad J_h^{1,0} \to \alpha^3, \quad J_h^{1,1} \to \alpha^2,
\end{align*}
where
\begin{align*}
	\psi_h = \frac{2 \sigma^\alpha}{r\alpha^2(2-\alpha)^{\frac{\alpha}{2}}} \frac{1}{h^{1-\frac{\alpha}{2}}\log(1/h)^\frac{\alpha}{2}}.
\end{align*}
Using furthermore that $v_h \to \frac{2}{\alpha(2-\alpha)}$, this yields 
\begin{align*}
	&\begin{pmatrix}
		\frac{r\alpha}{w_h^{\alpha} \sqrt{\psi_h}} & 0 \\ 0& \frac{1}{w_h^{\alpha}\log(1/w_h)\sqrt{\psi_h}} 
	\end{pmatrix}
	\begin{pmatrix}
	\mathcal{I}_h^{r,r} & \mathcal{I}_h^{r,\alpha} \\ \mathcal{I}_h^{r,\alpha} & \mathcal{I}_h^{\alpha,\alpha}
	\end{pmatrix} 
	\begin{pmatrix}
		\frac{r\alpha}{w_h^{\alpha} \sqrt{\psi_h}} & 0 \\ 
		0 & \frac{1}{w_h^{\alpha}\log(1/w_h)\sqrt{\psi_h}} 
	\end{pmatrix} \\
	&\longrightarrow\qquad \begin{pmatrix}
		\alpha^4 & \frac{\alpha^4}{2-\alpha} \\
		\frac{\alpha^4}{2-\alpha} & \frac{\alpha^4}{(2-\alpha)^2}
	\end{pmatrix}.
\end{align*}
Some straightforward manipulations show that 
\begin{align*}
	&\quad\frac{(h\log(1/h))^\frac{\alpha}{2}}{h} \begin{pmatrix}
			1 & 0 \\
			0 & \frac{1}{\log(1/h)}
		\end{pmatrix}
		\begin{pmatrix}
			\mathcal{I}_h^{r,r} & \mathcal{I}_h^{r,\alpha} \\ \mathcal{I}_h^{r,\alpha} & \mathcal{I}_h^{\alpha,\alpha}
		\end{pmatrix} 
		\begin{pmatrix}
			1 & 0 \\
			0 & \frac{1}{\log(1/h)}
		\end{pmatrix}
		\\
		\longrightarrow&
		\frac{2r}{\sigma^\alpha \alpha^2 (2-\alpha)^\frac{\alpha}{2}}
		\begin{pmatrix}
			\frac{\alpha^2}{r^2} & \frac{\alpha^4}{2-\alpha} \frac{1}{r\alpha} \frac{2-\alpha}{2\alpha} \\
			\text{sym} & \frac{(2-\alpha)^2}{4\alpha^2} \frac{\alpha^4}{(2-\alpha)^2}
		\end{pmatrix}
		\\
		&=\frac{2r}{\sigma^\alpha (2-\alpha)^\frac{\alpha}{2}}
		\begin{pmatrix}
					\frac{1}{r^2} & \frac{1}{2r} \\
					\frac{1}{2r} &  \frac{1}{4}
		\end{pmatrix}
\end{align*}
This limiting matrix is singular. The off-diagonal entry $\mathcal{I}_h^{\alpha,r}$ has not been considered by \cite{ait2012identifying}. 
\end{proof}

\begin{proof}[Proof of Proposition \ref{prop:clt-single}]
	Denote the true parameter by $\alpha_{0,m}$ and $r_{0,m}^\pm$, respectively. 
	By Lemma \ref{lem:expect-derivative}, we have as $n\to\infty$, $h=1/n\to 0$, 
	\begin{align*}
			\frac{1}{hu^\alpha_{m} \log u} \partial_{\alpha_m} \tilde{F}_n(\alpha_{m}) 
		&\to r_m^+\mathcal{J}_{\alpha_m}^+f(0) + r_m^-\mathcal{J}_{\alpha_m}^-f(0), \\
			\frac{1}{hu^{\alpha_m}} \partial_{r_m^\pm} \tilde{F}_n(r_m^\pm) 
		&\to \mathcal{J}_{\alpha_m}^\pm f(0).
	\end{align*}
	This convergence holds uniformly on compacts in $\Theta$. 
	The limits are positive because $r_m^++r_m^->0$ by the definition of $\Theta$, and $\mathcal{J}_{\alpha_m}^\pm f(0)>0$ by assumption. 
	Moreover, Lemma \ref{lem:moment-clt} also holds for $\tilde{F}_n$, i.e. \begin{align}
		nu_n^{-\alpha_1/2} \tilde{F}_n(\theta_0) \wconv  \mathcal{N}\left( 0, (r_1^+ \mathcal{J}_{\alpha_1} + r_1^-\mathcal{J}_{\alpha_1}) f^2(0) \right). \label{eqn:clt-single-1}
	\end{align} 
	Thus, the existence of a consistent sequence of estimators follows along the same lines as Lemma \ref{lem:estimating-consistency}. 
	
	For the central limit theorem, we use the mean value theorem to obtain, for a value $\tilde{\alpha}_m$ between $\alpha_{0,m}$ and $\hat{\alpha}_m$, 
	\begin{align*}
		0 = \tilde{F}_n(\hat{\alpha}_m) = \tilde{F}_n(\alpha_{0,m}) + \partial_{\alpha_m} \tilde{F}_n(\tilde{\alpha}_m) ( \hat{\alpha}_m - \alpha_{0,m}).
	\end{align*}
	In particular, $(\hat{\alpha}_m - \alpha_{0,m}) = - (\partial_{\alpha_m} \tilde{F}_n(\tilde{\alpha}_m))^{-1} \tilde{F}_n(\alpha_{0,m})$.
	Just as in the proof of Theorem \ref{thm:estimating-clt}, we may use the convergence of $\partial_{\alpha_m}\tilde{F}_n(\alpha_m)$ and the central limit theorem \eqref{eqn:clt-single-1} to derive the asymptotic distribution of $\hat{\alpha}_m$ by means of Slutsky's Lemma. 
	Analogously for $r_m^\pm$.
\end{proof}

\bibliography{BGbib}
\bibliographystyle{apalike}

\end{document}